\documentclass[11pt]{article}
\usepackage[utf8]{inputenc}
\usepackage{lmodern}
\usepackage{subfiles}
\usepackage{enumitem}
\setenumerate{topsep=6pt,ref={\normalfont(\roman*)},label={\normalfont(\roman*)}, itemsep=0pt} 

\usepackage{amsfonts}
\usepackage{amsthm}
\usepackage{amsmath}
\usepackage{amssymb}
\usepackage{amscd}
\usepackage{mathrsfs}
\usepackage{mathtools}
\usepackage{bbm}
\usepackage{esint}
\usepackage{slashed}

\usepackage[margin=3cm]{geometry}
\usepackage{setspace}
\usepackage{indentfirst}
\usepackage{graphicx}
\usepackage{graphics}
\usepackage{lscape}
\usepackage{pgf,tikz}
\usepackage{tikz-cd}
\usepackage{color}
\usepackage{pict2e}
\usepackage{epic}
\usepackage{epstopdf}
\usepackage{titlesec, titlefoot}
\titleformat{\section}[block]{\Large\bfseries\filcenter}{\thesection}{1em}{}
\usepackage{commath}
\usepackage{float}
\usepackage{caption}
\usepackage{etoolbox}
\usepackage[affil-it]{authblk}
\usepackage{combelow}

\usepackage[hidelinks, bookmarksdepth=3]{hyperref}
\hypersetup{bookmarksopen=true} 
\usepackage{hypcap}

\graphicspath{{./Pictures/}}
\allowdisplaybreaks

\expandafter\def\expandafter\normalsize\expandafter{%
\normalsize
\setlength\abovedisplayskip{6pt}
\setlength\belowdisplayskip{6pt}
\setlength\abovedisplayshortskip{6pt}
\setlength\belowdisplayshortskip{6pt}
}

\theoremstyle{plain}

\renewcommand*\thesection{\arabic{section}}
\numberwithin{equation}{section} 

\newtheorem{theorem}{Theorem}[section]
\newtheorem{lemma}[theorem]{Lemma}
\newtheorem*{lemma*}{Lemma}
\newtheorem{proposition}[theorem]{Proposition}
\newtheorem{corollary}[theorem]{Corollary}

\theoremstyle{definition}
\newtheorem{definition}[theorem]{Definition}
\newtheorem{remark}[theorem]{Remark}

\expandafter\let\expandafter\oldproof\csname\string\proof\endcsname
\let\oldendproof\endproof
\renewenvironment{proof}[1][\proofname]{%
\oldproof[\upshape \bfseries #1]%
}{\oldendproof}

\makeatletter
\def\@makechapterhead#1{%
\vspace*{50\p@}%
{\parindent \z@ \raggedright \normalfont
\interlinepenalty\@M
\Huge\bfseries  \thechapter.\quad #1\par\nobreak
\vskip 40\p@
}}
\makeatother

\renewcommand{\Re}{\operatorname{Re}}
\renewcommand{\Im}{\operatorname{Im}}

\def \a{\alpha}
\def \R {\mathbb{R}}
\def \Z {\mathbb{Z}}
\def \C{\mathbb{C}}

\def \N{\mathbb{N}}

\def \e{\varepsilon}

\def \p{\partial}
\def \mc{\mathcal}
\def \mb{\mathbb}

\def \tp{\textup}
\def \Id{\textup{Id}}
\def \id{\textup{id}}

\def \FS{\mathrm{FS}}
\newcommand{\eps}{\varepsilon}
\newcommand{\Om}{\Omega}

\DeclareMathOperator{\CP}{\mathbb{CP}}
\DeclareMathOperator{\B}{\mathbb{B}}

\DeclareMathOperator{\sing}{sing}
\DeclareMathOperator{\ddiv}{div}
\DeclareMathOperator{\supp}{supp}

\DeclareMathOperator{\tr}{tr}

\begin{document}

	\title{\textbf{Minimizing tangent maps\\ into complex projective space}}

	\author[1]{{\Large Matilde Gianocca}}
    \author[2,3]{{\Large Andr\'e Guerra}}
    \author[4]{{\Large David Ziener}}

    \affil[1]{\small Department of Mathematics, University of Chicago,\protect\\
    5734 S. University Avenue, Chicago, IL 60637, USA
    \protect\\
    {\tt{mgianocca@uchicago.edu}} \vspace{1em} \ }
	\affil[2]{\small Department of Pure Mathematics and Mathematical Statistics,  University of Cambridge,\protect\\  Wilberforce Rd, Cambridge CB3 0WB, UK
	\protect\\
	{\tt{adblg2@cam.ac.uk}} \vspace{1em} \ }
    \affil[3]{\small Departamento de Matem\'atica, Instituto Superior T\'ecnico,\protect\\
        Av. Rovisco Pais 1049-001 Lisboa, Portugal
        \protect\\
    {\tt{andre.l.guerra@tecnico.ulisboa.pt}} \vspace{1em} \ }
	\affil[4]{\small Max Planck Institute for Mathematics in the Sciences,\protect\\  Inselstraße 22, 04103 Leipzig, Germany
	\protect\\
	{\tt{david.ziener@mis.mpg.de}} \vspace{1em} \ }

    \date{}
	\maketitle
		
	\begin{abstract}
	We classify all nonconstant smooth harmonic maps $\mb S^2\to \CP^n$ and $\mb S^3\to \CP^n$ whose 0-homogeneous extensions are minimizing: they are, respectively, isometric embeddings or holomorphic isometric embeddings composed with a Hopf map.
	\end{abstract}

    \medskip


    \smallskip


    \vspace{1cm}

\section{Introduction}
In this paper we study the fine regularity properties of harmonic maps. Let $\Omega\subset \mathbb{R}^d$ be a domain and let $N\subset \mathbb{R}^L$ be a compact smooth manifold. When $d=2$, weakly harmonic maps are smooth \cite{Helein2002} but, as soon as $d\geq 3$, harmonic maps may be singular, even when $N$ is a sphere \cite{Lin2008}. 

For  stationary harmonic maps, the local structure of such singularities is described by \textit{tangent maps}. Indeed, the monotonicity formula implies that, after suitable rescalings around a singular point, a stationary harmonic map converges to a 0-homogeneous harmonic map
$$\hat u(r \omega):= u(\omega), \qquad   u\colon \mb S^{d-1}\to N, \qquad r\in (0,\infty), \omega\in \mb S^{d-1}.$$
Tangent maps provide models for the original map near its singular set, and thus their classification is a fundamental step toward understanding the structure of the possible singularities as well as the geometry of the singular set.

For maps with isolated singularities, that is when $u\in C^\infty(\mb S^{d-1};N)$, the corresponding tangent map is unique, i.e.\ it is independent of the choice of rescalings \cite{Simon1983}. Concerning the structure of the class of minimizing tangent maps  it is known that, if $\pi_1(N)$ is finite, then this family is compact in $W^{1,2}(\mb B^n;N)$, see \cite[Theorem 5.7]{Mazowiecka2018} and \cite[Theorem 6.4]{Hardt1987a}. This suggests that, at least in low dimensions, the collection of minimizing tangent maps may be small.

Nonetheless,  and despite being such a classical problem, the classification of non-trivial minimizing tangent maps remains very incomplete, and is only known in two instances:
\begin{enumerate}
\item\label{it:BCL} if $d=3$, $N=\mb S^2$, and if $\hat u$ is minimizing,  then $u\colon \mb S^2\to \mb S^2$ is an isometry \cite{Brezis1986};
\item \label{it:Naka} if $d=4$, $N=\mb S^3$, and if $\hat u$ is stable, then $u\colon \mb S^3\to \mb S^3$ is an isometry \cite{Nakajima2006}. 
\end{enumerate}

In this paper we are especially concerned with the case where $N=\CP^n$ is the complex projective space. This is the prototypical  example of a K\"ahler manifold and there is a long and fruitful interplay between harmonic maps and complex geometry, see e.g.\ the classical works \cite{Sampson1978,Siu1980a,Siu1980b}. From the point of view of regularity theory, the key relevant fact is that $\mathbb{CP}^{n}$ contains totally geodesic projective lines $\mathbb{CP}^{1}\cong \mb S^{2}$, which generates
\begin{equation}
\label{eq:pi2}
\pi_{2}(\mathbb{CP}^{n})\cong \mathbb{Z}.
\end{equation}
In particular smooth maps $u\colon \mb S^2\to \mathbb{CP}^{n}$ carry a natural notion of degree which acts as a topological obstruction to the existence of a smooth extension to $\mb B^3$.

In this paper, we present two new classification results for maps into $\CP^n$, for any $n\geq 1$. Our first theorem extends \ref{it:BCL} to all projective spaces:

\begin{theorem}
    \label{thm:Class3d}
    Let $u\in C^\infty(\mb S^2;\CP^n)$ be a nonconstant harmonic map.  If
    $\hat u$ is minimizing, then $u$ is an isometric embedding.
\end{theorem}

Our second classification result can be seen as an analogue of \ref{it:Naka}, although the geometry in our problem is substantially more involved. Starting with $n=1$ for simplicity, there is an analogue of \eqref{eq:pi2}, namely
$$\pi_3(\CP^1)\cong \Z.$$
Thus smooth maps $u\colon \mb S^3\to \CP^1$ carry a non-trivial degree, the so-called \textit{Hopf degree}, which again acts as an obstruction to the existence of smooth extensions to $\mb B^4$. In fact, the group $\pi_3(\CP^1)$ is generated by the \textit{Hopf map}
\begin{equation}
\label{eq:Hopfintro}
H\colon \mb S^3\subset \C^2 \to \CP^1, \qquad (z_0,z_1)\mapsto [z_0:z_1],
\end{equation}
whose extension $\hat H\colon \mb B^4\to \CP^1$ is minimizing by \cite{Coron1989}, see also \cite{Nichols1995}. Our next result shows that this is essentially the only minimizing tangent map:

\begin{theorem}
    \label{thm:Class4d}
    Let $u\in C^\infty(\mb S^3;\CP^n)$ be a nonconstant harmonic map.
    If $\hat u$ is minimizing, then there is an isometry $Q\colon\mb S^3\to\mb S^3$ and a holomorphic isometric embedding
    $\iota\colon\CP^1\to\CP^n$ such that
    \[
        u=\iota\circ H\circ Q.
    \]
\end{theorem}

We also prove that the maps in Theorem \ref{thm:Class4d} are actually minimizing.
\begin{theorem}
    \label{thm:HopfMin}
    Let $n\geq1$, let $Q\colon\mb S^3\to\mb S^3$ be an isometry, and let $\iota\colon\CP^1\to\CP^n$ be a holomorphic isometric embedding. If we write
    \[
        u=\iota\circ H\circ Q,
    \]
    then $\hat u$ is minimizing in $\mb B^4$.
\end{theorem}

We next remark that the minimality assumption is crucial in both Theorems \ref{thm:Class3d} and \ref{thm:Class4d}.

\begin{remark}[Necessity of minimality]\label{rmk:minim}
If one weakens minimality to the assumption that $\hat u$ is stable-stationary in Theorems \ref{thm:Class3d} and \ref{thm:Class4d} then there is a plethora of additional tangent maps. This is well-known in Theorem \ref{thm:Class3d} already for $n=1$, as any harmonic map $u\colon \mb S^2\to \CP^n$ is holomorphic or anti-holomorphic and hence has a stable extension; the stationarity condition imposes only a certain balancing condition, see already Propositions \ref{prop:stationarityHomogeneous} and \ref{prop:stabilitycharact}. In Theorem \ref{thm:Class4d}, also for $n=1$, all of the maps $u=\psi \circ H$, where $\psi\colon \mb S^2\to \mb S^2$ is a M\"obius map, have stable-stationary extensions. Thus, in this case, the situation is quite different from that in \ref{it:Naka}.
\end{remark}

\begin{remark}
    \label{rmk:modelmin}
    Minimality of the maps in Theorem \ref{thm:Class3d} was already established by Fujioka \cite{Fujioka1993}. Theorem \ref{thm:HopfMin} is particularly interesting when $n\geq2$: since $\pi_3(\CP^n)=0$, the boundary map $\iota\circ H$ is null-homotopic and admits a smooth extension to $\mb B^4$, whereas its $0$-homogeneous extension is a minimizer with singularities. This is similar to the situation for the equator maps $(x/|x|,0)\colon\mb B^m\to\mb S^m$ which are minimizing for $m\geq7$ \cite{JagerKaul1983}. Here the phenomenon occurs already in dimension four.
\end{remark}

The classification of minimizing tangent maps into $\CP^n$ was studied previously in \cite{Fujioka1993}. For minimality results for the Hopf map in variants of the harmonic map problem, see \cite{GuerraLamyZemas2025,Riviere1998}.

\medskip

The proofs of Theorems \ref{thm:Class3d}, \ref{thm:Class4d} and \ref{thm:HopfMin} are quite distinct, and we give here a brief description of the main ideas and techniques  that go into proving these results. 

The proof of Theorem \ref{thm:Class3d} relies fundamentally on the fact that all harmonic maps $\mb S^2\to \CP^n$ can be completely characterized \cite{Eells1983} using a combination of twistor theory and integrable system methods, available since the domain is two-dimensional. We do not intend here an overview of these methods, since the precise constructions play a very minor role in our arguments. Nonetheless, we mention the seminal work \cite{Uhlenbeck1989} where all harmonic two-spheres into the unitary group as well as into Grassmannians were constructed, and the monographs \cite{Burstall1990,Helein2001} for an exposition of this vast theory as well as further references. 

For the proof of Theorem \ref{thm:Class3d}, the relevant fact is that all harmonic maps $u\colon \mb S^2\to \CP^n$ are either holomorphic or anti-holomorphic, or else they have quite large energy (see Theorem \ref{thm:harmonicSphereDichotomyCPn}). Our argument is then based on explicitly constructing different competitors for the homogeneous extension $u$, depending on whether $u$ is non-holomorphic, holomorphic with degree at least 3, or holomorphic with degree 2, see Section \ref{sec:B3CPn}. 

The proof of Theorem \ref{thm:Class4d} is quite different. We first provide a non-trivial generalization of a recent argument due to Rivière \cite{Riviere2023} to prove that, if $u\colon \mb S^3\to \CP^n$ is smooth and $\hat u$ is stable, then $u$ is \textit{pseudo-horizontally conformal} (see Definition \ref{def:HWC} and Theorem \ref{thm:Riv}). The study of stability of harmonic maps into $\CP^n$ is somewhat related to this calculation, cf.\ \cite{Burns1989,Chen1996,Lawson1973}. A transversality argument then allows us to essentially reduce to the case $n=1$, in which case we infer that $u$ is a harmonic morphism. The theory of harmonic morphisms \cite{Baird2003}, which again relies on twistor methods, then asserts that up to isometries we have $u=f\circ H$, where $f\colon \CP^1\to \CP^n$ is holomorphic. A key difficulty is showing that $f$ is isometric, as this cannot be inferred from stability alone, cf.\ Remark \ref{rmk:minim}, but we manage to infer it from Theorem \ref{thm:Class3d} by using the constant dilation property of the Hopf map.

In the final stage of completing this manuscript the author's became aware of the recent independent work of Lara \cite{Lara2026}, which contains some results related to Theorem \ref{thm:Riv}.

Theorem \ref{thm:HopfMin} shows that the cone over the Hopf map $H$ is minimizing as a map into $\CP^n$. The minimality in the case $n=1$ was shown by Coron and Gulliver \cite{Coron1989}. The strategy of the proof for general $n$ is to study a given competitor through a family of projective deformations, such that in the weak limit the map takes values into a projective line and the Theorem can be reduced to the $n=1$ case.
 
We conclude this discussion by mentioning a possible extension of Theorem \ref{thm:Class4d}:

\begin{remark}[Cylindrical maps]\label{rmk:cylinders}
    In Theorem \ref{thm:Class4d}, the assumption that $u\colon \mb S^3\to \CP^n$ is smooth cannot be omitted: there is at least another family of minimizing tangent maps, namely the cylindrical extension of the maps in Theorem \ref{thm:Class3d}. We believe that, at least when $n=1$, this should be the only other family of minimizing tangent maps. The key difficulty is that we cannot infer that a singular map $u$ must be a harmonic morphism, for otherwise we could apply the classification of singular harmonic morphisms from \cite{Baird1991}. Indeed, as we explain in Remark \ref{rem:stabsing}, it seems impossible to apply stability considerations to study minimizing tangent maps with non-isolated singularities.
\end{remark}

\subsubsection*{Acknowledgements}

A.G. was supported by a Royal Society Newton International Fellowship. M.G. and D.Z. gratefully acknowledge the hospitality of the University of Cambridge, where part of this research was conducted. Their visits were funded through A.G.’s Royal Society Newton International Fellowship. The authors thank Tristan Rivi\`ere for his interest in this work and for discussions of his low-index classification \cite{Riv23}. D.Z. thanks his supervisor, László Székelyhidi, for his continued support and guidance.

\subsubsection*{AI disclosure}
This project began roughly one year ago without any AI assistance. The central ideas and strategies used in the proofs are due to the human authors. As the usefulness of AI tools for mathematical research increased, AI tools in particular OpenAI's ChatGPT 5.6 Sol, was incorporated in the research and used to assist with computations and to simplify some proofs. In particular, the energy computation in the proof of Proposition \ref{prop:step3} were carried out by ChatGPT 5.6 Sol, after the author's suggested to work in the prolate spheroidal coordinates and to simplify the computations using the stationarity of the map. Moreover, the simplification which relies on using the same profile function $t_\eps$ in the proofs of Propositions \ref{prop:step1}--\ref{prop:step2} is also due to the same AI model. The paper was written and checked by the human authors who take full responsibility for its correctness. 

\clearpage

\section{Preliminary results on stationary harmonic maps}
\label{sec:prelims}

\subsection{Notation and definitions}

 We recall first some basic definitions which the reader can find in \cite{Lin2008,Simon1996}. Let $(M,g)$ and $(N,h)$ be
Riemannian manifolds, assume that $N$ is compact and isometrically embedded
in $\R^L$,
and denote by $R^N$ the Riemann curvature tensor of $N$.
 
 For an open set
$U\subset M$, write
\[
E_U(u):= \frac12\int_U |d u|_g^2\,dV_g .
\]
Whenever $(e_\alpha)$ is a local orthonormal frame on $M$, we write $u_\alpha:=du(e_\alpha)$, thus
$|\nabla u|_g^2=\sum_\alpha h(u_\alpha,u_\alpha)$.  
We consider, as usual, the Sobolev space
\[
W^{1,2}(M;N):=\{u\in W^{1,2}(M;\R^L):u(x) \in N
\text{ for a.e.\ } x\in M\}.
\]
For $u\in W^{1,2}(M;N)$, its pullback tangent bundle is
\[
    u^*TN
    :=\{(x,V)\in M\times\R^L:V\in T_{u(x)}N\}.
\]
We define the spaces of sections
\begin{align*}
    \mc V(u)
    &:={\left\{\phi\in W^{1,2}(M;\R^L)\cap L^\infty(M;\R^L):
          \phi(x)\in T_{u(x)}N\text{ for a.e.\ }x\in M\right\}},\\
    \mc V_c(u)
    &:={\left\{\phi\in\mc V(u):\supp(\phi)\Subset M\right\}}.
\end{align*}
For $\phi\in\mc V(u)$, its weak covariant derivative is
$\nabla\phi:=P_u(d\phi)$, where
$P_u\colon\R^L\to T_{u(x)}N$ is the orthogonal projection.

A map $u\in W^{1,2}(M;N)$ is said to be a \textit{weakly harmonic map}
if it is a critical point of $E_M$ with respect to compactly supported
outer variations.  Equivalently, $u$ satisfies 
\begin{equation}
    \label{eq:weakHarmonicMap}
    \int_M\sum_{\alpha=1}^{\dim M}
    h\bigl(u_\alpha,\nabla_{e_\alpha}\phi\bigr)\,dV_g=0
    \qquad\text{for every }\phi\in\mc V_c(u),
\end{equation}
where $(e_\alpha)$ is any local orthonormal frame.   For
a weakly harmonic map $u$, we say that $u$ is:
\begin{enumerate}
\item \textit{minimizing on compact subdomains} if, for every open $U\Subset M$,
\[
E_U(u)\leq E_U(v)
\]
for all $v\in W^{1,2}(U;N)$ such that
$u-v\in W^{1,2}_0(U;\mb{R}^{L})$;

\item \textit{stationary} if it is a critical point of $E_M$ with respect to
compactly supported inner variations.  More precisely, if $\Phi_t^X$ denotes
the flow of $X\in C_c^\infty(TM)$, then
\begin{equation}
    \label{eq:stationarityInnerVariations}
    \left.\frac{d}{dt}\right|_{t=0}E_M(u\circ\Phi_t^X)=0
    \qquad\text{for every }X\in C_c^\infty(TM).
\end{equation}
Equivalently,  we have
\begin{equation}
    \label{eq:statM}
    0=\int_M S_{\alpha\beta}(u)\nabla^\alpha X^\beta\,dV_g
    =\int_M\left(
    \frac12|\nabla u|_g^2\,\ddiv_gX
    -h(u_\alpha,u_\beta)
      \nabla^\alpha X^\beta
    \right)dV_g
\end{equation}
where  we define the stress-energy tensor
\[
S_{\alpha\beta}(u)
:=\frac12|\nabla u|_g^2g_{\alpha\beta}
-h(u_\alpha,u_\beta);
\]

\item \textit{stable} if, for every
$\phi\in \mc V_c(u)$, we have
\begin{equation}
\label{eq:stabilitydef}
\begin{aligned}
0\leq E''_{u,M}(\phi)
&:=\int_M
\left(
|\nabla\phi|_g^2
-\sum_{\alpha=1}^{\dim M}
h\bigl(R^N(\phi,u_\alpha)u_\alpha,\phi\bigr)
\right)
dV_g\\
&=\int_M h(J_u\phi,\phi)\,dV_g,
\end{aligned}
\end{equation}
where $J_u$ is the \textit{Jacobi operator}
\begin{equation}
    \label{eq:JacobiOperator}
    J_u\phi
    :=\nabla^*\nabla\phi
      -\sum_{\alpha=1}^{\dim M}
       R^N(\phi,u_\alpha)u_\alpha.
\end{equation}
Here $\nabla^*\nabla$ is the nonnegative connection Laplacian.  At a
point where $(e_\alpha)$ is parallel,
\begin{equation}
    \label{eq:ConnectionLaplacian}
    \nabla^*\nabla\phi
    =-\sum_{\alpha=1}^{\dim M}
      \nabla_{e_\alpha}\nabla_{e_\alpha}\phi.
\end{equation}
\end{enumerate}
We also note that any minimizing map is stable and stationary \cite{Lin2008}.

In this paper we will only consider the cases where $M=\R^{m+1}$ or $M=\mb S^{m}$, and therefore we will often suppress the dependence on $g$ in the relevant expressions. Instead, we will use the standard symbols of differential operators for derivatives in $\R^{m+1}$, and subscripts to denote derivatives in $\mb S^m$.

\subsection{Stationarity and stability for tangent maps}

In this subsection we characterize stationarity and stability for tangent maps.  Recall that, for a map $u\colon \mb S^{m}\to N$, we denote its
0-homogeneous extension by $\hat u\colon \R^{m+1}\to N$. Here
$\langle\cdot,\cdot\rangle$ denotes the Euclidean inner product on
$\R^{m+1}$. It will be useful to consider, for $a\in\R^{m+1}$, the
conformal Killing fields
\begin{equation}
\label{eq:conf_killing}
f_a(x):=\langle a,x\rangle,
\qquad
X_a:=\nabla_{\mb S^m}f_a= a-\langle a,x\rangle x.
\end{equation}
It is also easy to verify the identity
\begin{equation}
\label{eq:hessianfa}
\nabla_{\mb S^m}^2f_a=-f_ag_{\mb S^m}.
\end{equation}
If $(e_k)_{k=1}^{m+1}$ is the standard basis of $\R^{m+1}$, for 
$v,w\in T_x\mb S^m$ we have
\begin{equation}
    \label{eq:ParsevalConformalFields}
    \sum_{k=1}^{m+1}
      g_{\mb S^m}(X_{e_k},v)g_{\mb S^m}(X_{e_k},w)
    =\sum_{k=1}^{m+1}\langle e_k,v\rangle\langle e_k,w\rangle
    =g_{\mb S^m}(v,w).
\end{equation}
Taking the trace of this identity gives
\begin{equation}
    \label{eq:ConformalFieldSums}
    \sum_{k=1}^{m+1}|X_{e_k}|^2=m,
    \qquad
    \sum_{k=1}^{m+1}f_{e_k}^2=|x|^2 = 1.
\end{equation}

To the best of our knowledge the following characterization of stationarity was first observed by Hardt for $m=2$, and then used in \cite{Ding2003} to construct non-stationary maps which are weak limits of stationary maps; the result for $m\geq 3$ appears to be novel, and it simplifies some of the arguments in \cite{ElSoufi1995,Struwe2020}, which are applicable only for smooth harmonic maps.

\begin{proposition}[Stationarity for tangent maps]
\label{prop:stationarityHomogeneous}
Let $m\geq2$ and let $u\in W^{1,2}(\mb S^m;N)$.   Then $\hat u$ is stationary in
$\R^{m+1}$ if and only if $u$ is stationary on $\mb S^m$ and satisfies
\begin{equation}
    \label{eq:balancingCondition}
    \int_{\mb S^m}|du|_{\mb S^m}^2\,\omega\,dV_{\mb S^m}(\omega)=0.
\end{equation}
If $m\geq3$, condition \eqref{eq:balancingCondition} follows automatically
from the stationarity of $u$.  
\end{proposition}

\begin{proof}
Let $p\colon\R^{m+1}\setminus\{0\}\to\mb S^m$ be the radial
projection.  Since
$g_{\R^{m+1}}=dr^2+r^2g_{\mb S^m}$ and $\partial_r\hat u=0$, the
stress-energy tensors are related by
\begin{equation}
    \label{eq:stressHomogeneousSplitting}
    S_{\R^{m+1}}(\hat u)
    =p^*S_{\mb S^m}(u)
      +\frac{1}{2r^2}|du|_{\mb S^m}^2\,dr\otimes dr.
\end{equation}
A direct computation in spherical coordinates shows
that $S_{\R^{m+1}}(\hat u)$ is weakly divergence-free on
$\R^{m+1}\setminus\{0\}$ if and only if
$S_{\mb S^m}(u)$ is weakly divergence-free on $\mb S^m$.  Thus stationarity
of $\hat u$ implies stationarity of $u$.

Suppose now that $u$ is stationary.  For
$X\in C_c^\infty(\R^{m+1};\R^{m+1})$, using \eqref{eq:statM}, \eqref{eq:stressHomogeneousSplitting}, and integrating by parts gives
\begin{equation}
    \label{eq:stationarityDefect}
    \int_{\R^{m+1}\setminus B_\varepsilon}
      S_{\alpha\beta}(\hat u)\,\partial^\alpha X^\beta\,dx
    =-\frac12\varepsilon^{m-2}
      \int_{\mb S^m}|du|_{\mb S^m}^2
      \langle X(\varepsilon\omega),\omega\rangle
      \,dV_{\mb S^m}(\omega).
\end{equation}
If $m\geq3$, the right-hand side tends to zero as
$\varepsilon\to0$.  If $m=2$, it tends to
\[
    -\frac12\left\langle
      X(0),
      \int_{\mb S^2}|du|_{\mb S^2}^2\,\omega\,dV_{\mb S^2}(\omega)
    \right\rangle.
\]
It follows that, for $m=2$, the extension $\hat u$ is stationary if and
only if \eqref{eq:balancingCondition} holds.

It remains to verify that, when $m\geq 3$, stationarity implies \eqref{eq:balancingCondition}.  Using $X_a$
from \eqref{eq:conf_killing} as a test vector field in
\eqref{eq:statM}, together with \eqref{eq:hessianfa}, we obtain
\[
    0
    =\int_{\mb S^m}
      S_{\alpha\beta}(u)\nabla^\alpha X_a^\beta\,dV_{\mb S^m}
    =-\int_{\mb S^m}
      f_a\,\tr_{g_{\mb S^m}}S_{\mb S^m}(u)\,dV_{\mb S^m}.
\]
Since
\[
    \tr_{g_{\mb S^m}}S_{\mb S^m}(u)
    =\frac{m-2}{2}|du|_{\mb S^m}^2
\]
and $m-2\neq0$, this implies
\[
    0=\int_{\mb S^m}f_a|du|_{\mb S^m}^2\,dV_{\mb S^m}
    =\left\langle
      a,\int_{\mb S^m}|du|_{\mb S^m}^2\,\omega
      \,dV_{\mb S^m}(\omega)
    \right\rangle.
\]
Since $a\in\R^{m+1}$ is arbitrary, \eqref{eq:balancingCondition} follows.
\end{proof}

The next lemma shows that stationary harmonic maps on $\mb S^m, m\geq 3$, are not local minima of the energy. Although this result is well-known for smooth harmonic maps \cite{Xin1980}, it does not appear to have been observed for stationary ones, and the proof below is somewhat different from the usual proofs in the smooth setting.

\begin{lemma}[Instability of stationary maps on $\mb S^m$]
\label{lem:ConformalInstability}
Let $m\geq3$ and let $u\in W^{1,2}(\mb S^m;N)$ be a nonconstant
stationary harmonic map.  Then there is a one-parameter group of
conformal diffeomorphisms
$\phi_t\colon\mb S^m\to\mb S^m$ such that
\[
    \left.\frac{d^2}{dt^2}\right|_{t=0}E_{\mb S^m}(u\circ\phi_t)<0.
\]
\end{lemma}

\begin{proof}
Let $\phi_t^a$ be the flow of $X_a$ defined in \eqref{eq:conf_killing}.
By \eqref{eq:hessianfa},
$\mathcal L_{X_a}g_{\mb S^m}=2\nabla_{\mb S^m}^2f_a
=-2f_ag_{\mb S^m}$.  Thus $\phi_t^a$ is
conformal, say
\[
    (\phi_t^a)^*g_{\mb S^m}=e^{2\omega_t^a}g_{\mb S^m},
    \qquad \omega_0^a=0,
\]
and differentiation with respect to $t$ gives
\[
    \partial_t\omega_t^a=-f_a\circ\phi_t^a.
\]
Consequently, if $\sigma_t^a:=\omega_t^a\circ(\phi_t^a)^{-1}$, then $\p_t \sigma_t^a = -f_a - X_a(\sigma_t^a)$ and 
\[
    \left.\partial_t\right|_{t=0}\sigma_t^a=-f_a,
    \qquad
    \left.\partial_t^2\right|_{t=0}\sigma_t^a=X_a(f_a)
    =|\nabla_{\mb S^m}f_a|^2.
\]

Since $\phi_t^a$ is conformal, changing variables in the energy, we obtain
\[
    E_{\mb S^m}(u\circ\phi_t^a)
    =\frac12\int_{\mb S^m}
      e^{(2-m)\sigma_t^a}
      |du|^2\,dV_{\mb S^m}.
\]
By stationarity \eqref{eq:stationarityInnerVariations}, the first
derivative at $t=0$ vanishes.  The preceding identities give
\[
\left.\frac{d^2}{dt^2}\right|_{t=0}E_{\mb S^m}(u\circ\phi_t^a)
{}=\frac{m-2}{2}\int_{\mb S^m}
 \left((m-2)f_a^2-|\nabla_{\mb S^m}f_a|^2\right)
 |du|^2\,dV_{\mb S^m}.
\]
Using \eqref{eq:ConformalFieldSums} and summing over $k$ gives
\[
    \sum_{k=1}^{m+1}
    \left.\frac{d^2}{dt^2}\right|_{t=0}
      E_{\mb S^m}(u\circ\phi_t^{e_k})
    =-(m-2)\int_{\mb S^m}|du|^2\,dV_{\mb S^m}<0.
\]
The last inequality follows because $u$ is nonconstant.  Hence at least
one of the conformal flows $\phi_t^{e_k}$ has strictly negative second
variation.
\end{proof}


In light of Lemma \ref{lem:ConformalInstability}, it is natural to look for a measure of instability of $u$, namely
\begin{equation}
    \label{eq:deflambdaone}
    \lambda_1(u)
    :=\inf\left\{
    \frac{E''_{u,\mb S^m}(\eta)}{\|\eta\|_{L^2(\mb S^m)}^2}:
    0\neq\eta\in\mc V(u)
    \right\}.
\end{equation}
Due to the Rayleigh quotient characterization of eigenvalues, when $u\in C^\infty(\mb S^m;N)$ the quantity $\lambda_1(u)$ is simply the smallest eigenvalue of the Jacobi operator, which in this case is a Fredholm operator, see e.g.\ \cite[Proposition 6.1]{Hardt1992a}. In this case, Lemma \ref{lem:ConformalInstability} leads to the following classical theorem:

\begin{theorem}[{\cite{Xin1980}}]\label{thm:Xin}
Let $u\in C^\infty(\mathbb S^m;N)$, $m\geq 3$, be a
non-constant harmonic map. With the notation \eqref{eq:conf_killing}, we have
\begin{equation}
\label{eq:Xin}
J_u \,du(X_a)=(2-m)du(X_a).
\end{equation}
In particular, there exists
$a\in\mathbb R^{m+1}$ such that $du(X_a)\neq 0$, and therefore
\[
\lambda_1(u)\leq 2-m .
\]
\end{theorem}

The next result asserts that if $u\colon \mb S^m\to N$ is the restriction of a stable  0-homogeneous map then, although  it is unstable as a map on the sphere by Theorem \ref{thm:Xin}, nonetheless it cannot be \textit{too} unstable:

\begin{proposition}[Stability for tangent maps]\label{prop:stabilitycharact}
Let $m\geq2$ and let
$\hat u\in W^{1,2}_\tp{loc}(\R^{m+1};N)$ be a 0-homogeneous harmonic map.
If $\hat u$ is stable, then
\begin{equation}
    \label{eq:lowerboundlambda}
    - \frac{(m-1)^2}{4} \leq \lambda_1(u).
\end{equation}
Conversely, if \eqref{eq:lowerboundlambda} holds and
$\mc H^{m-1}(\sing(\hat u))=0$, then $\hat u$ is stable.
\end{proposition}

The necessity of \eqref{eq:lowerboundlambda} for stability is well-known, see e.g.\ \cite{Schoen1984} or \cite[Lemma 4.2]{Hardt1992a}, but for the reader's convenience we present a complete proof.

\begin{proof}
Write $x=r\omega$ in spherical coordinates. Since $\hat u$ is 0-homogeneous, we have
\begin{align}
    E''_{\hat u,\R^{m+1}}(\phi)
     =\int_0^\infty r^m
       \|\partial_r\phi(r,\cdot)\|_{L^2(\mb S^m)}^2\,dr
    +\int_0^\infty r^{m-2}
       E''_{u,\mb S^m}(\phi(r,\cdot))\,dr.        \label{eq:sphericalJacobi}
\end{align}

Suppose first that $\hat u$ is stable and consider 
\begin{equation*}
    \phi(r\omega)=\psi(r)\eta(\omega),
\end{equation*}
where $\psi\in C_c^\infty((0,\infty))$ and
$\eta\in\mc V(u)$.  Substituting this into
\eqref{eq:sphericalJacobi} gives
\begin{equation}
    \label{eq:separatedJacobi}
    E''_{\hat u,\R^{m+1}}(\phi)
    =\left(\int_0^\infty |\psi'|^2r^m\,dr\right)
       \|\eta\|_{L^2(\mb S^m)}^2
     +\left(\int_0^\infty \psi^2r^{m-2}\,dr\right)
       E''_{u,\mb S^m}(\eta).
\end{equation}
The sharp Hardy inequality gives
\begin{equation}
    \label{eq:radialHardy}
    \inf_{0\neq\psi\in C_c^\infty((0,\infty))}
    \frac{\int_0^\infty |\psi'|^2r^m\,dr}
         {\int_0^\infty \psi^2r^{m-2}\,dr}
    =\frac{(m-1)^2}{4};
\end{equation}
see \cite[Lemma 1.3]{Schoen1984}.  Dividing
\eqref{eq:separatedJacobi} by the second radial integral, optimizing in
$\psi$, and then taking the infimum over $\eta$ proves
\eqref{eq:lowerboundlambda}.

Conversely, assume
$\lambda_1(u)\geq-(m-1)^2/4$.  First take
\[
    \phi\in\mc V_c(\hat u),
    \qquad
    \supp(\phi)\Subset\R^{m+1}\setminus\sing(\hat u).
\]
For almost every $r$, the section $\phi(r,\cdot)$ belongs to $\mc V(u)$ and is therefore an admissible test field in
\eqref{eq:deflambdaone}.  Consequently,
\begin{equation*}
    E''_{\hat u,\R^{m+1}}(\phi)
    \geq
    \int_0^\infty r^m
       \|\partial_r\phi(r,\cdot)\|_{L^2}^2\,dr
    -\frac{(m-1)^2}{4}\int_0^\infty r^{m-2}
       \|\phi(r,\cdot)\|_{L^2}^2\,dr .
\end{equation*}
Applying \eqref{eq:radialHardy} to each Euclidean
component of $\phi(\cdot,\omega)\in\R^L$,
summing over the components and integrating over $\mb S^m$ gives
\[
    \int_0^\infty r^m
       \|\partial_r\phi(r,\cdot)\|_{L^2}^2\,dr
    \geq \frac{(m-1)^2}{4}\int_0^\infty r^{m-2}
       \|\phi(r,\cdot)\|_{L^2}^2\,dr.
\]
This proves the stability
inequality for fields supported away from $\sing(\hat u)$.

It remains to remove this restriction.  Let $\phi\in\mc V_c(\hat u)$.
Since $\mc H^{m-1}(\sing(\hat u))=0$, $\sing(\hat u)$ has zero 2-capacity \cite{Evans2015}, hence there are smooth cutoffs
$0\leq\chi_j\leq1$, vanishing on
a neighbourhood of $\sing(\hat u)\cap\supp(\phi)$, such that $\chi_j\to1$
almost everywhere and $\|\nabla\chi_j\|_{L^2}\to0$ on a fixed
neighbourhood of $\supp(\phi)$.  In particular,
\[
    \chi_j\phi\in\mc V_c(\hat u),
    \qquad
    \supp(\chi_j\phi)\Subset\R^{m+1}\setminus\sing(\hat u),
\]
so the preceding argument gives
$E''_{\hat u,\R^{m+1}}(\chi_j\phi)\geq0$.  Since $\phi$ is bounded,
$\chi_j\phi\to\phi$ in $W^{1,2}$.  Since $N$ is compact, dominated
convergence also gives
\[
\sum_\alpha
h\bigl(R^N(\chi_j\phi,\hat u_\alpha)\hat u_\alpha,\chi_j\phi\bigr)
\longrightarrow
\sum_\alpha
h\bigl(R^N(\phi,\hat u_\alpha)\hat u_\alpha,\phi\bigr)
\quad\text{in }L^1.
\]
Passing to the limit proves stability.
\end{proof}

We conclude this section with some remarks on the meaning of $\lambda_1(u)$ in the singular case:
\begin{remark}[Stability for singular maps]\label{rem:stabsing} 
Let us take $m\geq 3$. If $\hat u\in W^{1,2}_\tp{loc}(\R^{m+1};N)$ is minimizing, then
from \cite{SchoenUhlenbeck1982} we have $\dim_H(\sing\hat u)\leq m-2$. In particular Proposition \ref{prop:stabilitycharact} applies. However, the meaning of $\lambda_1(u)$ is hard to ascertain. 

Firstly, in general, it is not clear if $J_u$ is a Fredholm operator. Indeed, for a general weight $w\in L^1(\mb S^m)$, the embedding $W^{1,2}(\mb S^m)\subset L^2(w)$ is compact if $w\in L^{m/2}(M)$ \cite[Example 3.19]{Girouard2021}. For the Jacobi operator the relevant weight is $w=|d u|^2$, which is in $L^{m}$ if and only if $u$ is smooth \cite[Theorem I.2]{Bethuel1993}.

Secondly, since $\hat u$ is minimizing, by Proposition
\ref{prop:stationarityHomogeneous} we necessarily have that $u$ is
stationary and, therefore, by Lemma \ref{lem:ConformalInstability}, it is
energy-unstable.  Nonetheless, it can happen that $\lambda_1(u)\geq0$.
Indeed, write $x=(x',x_4)\in\R^3\times\R$ and consider the cylindrical minimizing tangent map
\[
    \hat u_*(x',x_4):=\frac{x'}{|x'|}\colon\mb B^4\to\mb S^2,
    \qquad
    u_*:=\hat u_*|_{\mb S^3}.
\]
Use coordinates $(\theta,\omega)\in(0,\pi)\times\mb S^2$ on
$\mb S^3$ given by
\[
    x'=\sin\theta\,\omega,
    \qquad x_4=\cos\theta.
\]
Then $u_*(\theta,\omega)=\omega$ and
$g_{\mb S^3}=d\theta^2+\sin^2\theta\,g_{\mb S^2}$.  Since the round
$\mb S^2$ has sectional curvature one, for every
$\eta\in\mc V(u_*)$ the second variation is
\begin{equation}
    \label{eq:cylindricalJacobiForm}
    E''_{u_*,\mb S^3}(\eta)
    =\int_0^\pi\int_{\mb S^2}
      \left(
        \sin^2\theta\,|\nabla_{\partial_\theta}\eta|^2
        +|\nabla_{\mb S^2}\eta|^2-|\eta|^2
      \right)dV_{\mb S^2}\,d\theta.
\end{equation}
Recall the sharp Poincaré inequality for tangent vector fields on $\mb S^2$,
\[
    \int_{\mb S^2}|\nabla_{\mb S^2}Y|^2\,dV_{\mb S^2}
    \geq\int_{\mb S^2}|Y|^2\,dV_{\mb S^2},
\]
and so we see that \eqref{eq:cylindricalJacobiForm} is nonnegative,
hence $\lambda_1(u_*)\geq0$.  On the other hand, if $Y$ is a 
Killing field on $\mb S^2$ and we set $\eta(\theta,\omega):=Y(\omega)$ then $E''_{u_*,\mb S^3}(\eta)=0$.  We conclude that
\[
    \lambda_1(u_*)=0.
\]

The above example suggests that, to study tangent maps with non-isolated singularities, stability is not a useful condition. This situation should be contrasted with that of singular minimal hypersurfaces, where the singular set has much higher codimension \cite{Zhu2018}.
\end{remark}

\clearpage

\section{The geometry of \texorpdfstring{$\CP^n$}{CPn}}

\subsection{Notation and definitions}

We begin by fixing some conventions that will be used
throughout the paper.  The complex projective space is the quotient
\[
    \CP^n=(\C^{n+1}\setminus\{0\})/\C^*,
\]
where $\C^*$ acts by scalar multiplication.  We write
$[z^0:\cdots:z^n]$ for the equivalence class of
$z\in\C^{n+1}\setminus\{0\}$; these are the homogeneous coordinates of a
point.  Let
\[
    U_j:=\{[z^0:\cdots:z^n]\in\CP^n:z^j\neq0\}, \qquad j=0,\dots, n.
\]
On $U_j$ the ratios $z^\alpha/z^j$, with $\alpha\neq j$, are holomorphic
coordinates, and on $U_0$ we write
\begin{equation}
    \label{eq:AffineCoordinatesCPn}
    w^\alpha=\frac{z^\alpha}{z^0},\qquad \alpha=1,\ldots,n,
    \qquad
    [z^0:\cdots:z^n]=[1:w^1:\cdots:w^n].
\end{equation}
The standard action of $\tp{U}(n+1)$ on $\C^{n+1}$ descends to
a transitive, holomorphic action on $\CP^n$, which moreover is isometric for the
Fubini--Study metric, which we now define.  In the coordinates $w=(w^1,\ldots,w^n)$ on $U_0$, the
function
\begin{equation}
    \label{eq:FubiniStudyPotential}
    \rho_\FS(w):=2\log(1+|w|^2)
\end{equation}
is a K\"ahler potential.  The corresponding Hermitian
form is
\begin{equation}
    \label{eq:FubiniStudyCoordinates}
    H_\FS
    =2\frac{\partial^2\rho_\FS}
        {\partial w^\alpha\partial\overline w^\beta}
      \,dw^\alpha\otimes d\overline w^\beta
    =4\frac{(1+|w|^2)\delta_{\alpha\beta}
        -\overline w^\alpha w^\beta}
        {(1+|w|^2)^2}
      \,dw^\alpha\otimes d\overline w^\beta,
\end{equation}
where as usual repeated indices  are summed.
The real part $h_\FS:=\Re H_\FS$ is the Fubini--Study Riemannian metric.
This, together with the standard complex structure $I$ inherited from $\C^n$,  makes $\CP^n$ a
K\"ahler manifold, with K\"ahler form 
\begin{equation}
    \label{eq:kahler}
    \omega_\FS(X,Y) =h_\FS(I X,Y),
    \qquad
    \omega_\FS:=i\,\partial\bar\partial\rho_\FS.
\end{equation}

We emphasize the presence of the factor $4$ in \eqref{eq:FubiniStudyCoordinates}: this normalization is so that 
the holomorphic sectional curvature is $1$, and every projective line $\CP^1\subset \CP^n$ with
the induced metric is a round two-sphere of curvature $1$ and area $4\pi$.
This normalization is also natural in view of the Hopf map, which  is the map
\[
    H_n\colon\mb S^{2n+1}\longrightarrow\CP^n,
    \qquad
    (z^0, \dots, z^n)\mapsto [z^0:\cdots:z^n].
\]
We usually suppress the subscript and write simply $H$.  The fibres of $H_n$ are circles, i.e.\
$    H^{-1}(H(z))=\{e^{i\theta}z:\theta\in\R\}.$
Write $\langle\cdot,\cdot\rangle$ for the standard Hermitian inner
product on $\C^{n+1}$.  The round metric is given by
\begin{equation}
    \label{eq:RoundMetricHermitianProduct}
    g_{\mb S^{2n+1}}
    =\left.\operatorname{Re}\langle\cdot,\cdot\rangle
      \right|_{T\mb S^{2n+1}}.
\end{equation}
The \textit{vertical} and \textit{horizontal spaces} of $dH_z$ are
\begin{equation}
    \label{eq:HopfVerticalHorizontal}
    \begin{aligned}
        \mc V_z&:=\ker dH_z=\R iz,\\
        \mc H_z&:=(\ker dH_z)^\perp=(\C z)^\perp,
    \end{aligned}
\end{equation}
where the orthogonal complement is taken with respect to $\langle\cdot, \cdot\rangle$.
Thus $\mc H_z$ is the orthogonal complement of $\mc V_z$ inside
$T_z\mb S^{2n+1}$ with respect to the round metric.  In particular,
if $V\in T_z\mb S^{2n+1}$, then its horizontal component is
\begin{equation}
    \label{eq:HopfHorizontalProjection}
    V^{\mathrm{hor}}=V-\langle V,z\rangle z.
\end{equation}
For every horizontal vector $V$,
\begin{equation}
    \label{eq:HopfDilation}
    |dH_z(V)|_{h_\FS}=2|V|.
\end{equation}
Thus, in this normalization, the Hopf map is horizontally conformal with constant dilation $2$, and this will play an important role in the proof of Lemma \ref{lem:ProjectionEnergy}.

When no confusion can arise, we drop the subscript $\FS$ from $h$ and
$\omega$.  With the curvature convention used in Section~2, the
curvature tensor of $h_\FS$ is
$$
R(X,Y)Z= \frac 1 4\left(h(Y,Z)X-h(X,Z)Y+ h(IY,Z) IX - h(IX,Z)IY + 2 h(X,IY)IZ\right),
$$
cf.\ \cite[(1)]{Chen1996},
and so, taking $Z=Y$ and using that $h(I\cdot,I\cdot)=h$, we have
\begin{equation}
\label{eq:curvsimple}
R(X,Y)Y=\frac 1 4 \left(|Y|_h^2 X-h(X,Y) Y + 3 h(X,IY) IY\right).
\end{equation}
This identity will play an important role in Section \ref{sec:B4CPn}.

With our normalization, $[\omega_\FS/(4\pi)]$ is the positive generator of
$H^2(\CP^n;\Z)\cong \Z$.  For every map $u\in C^\infty(\CP^1;\CP^n)$, we define its
\emph{degree} $d(u)\in\Z$ by
\[
    u^*\left[\frac{\omega_\FS}{4\pi}\right]
    =d(u)\left[\frac{\omega_\FS}{4\pi}\right]
    \quad\text{in }H^2(\CP^1;\Z),
\]
where the form on the right is the Fubini--Study form of $\CP^1$.
Equivalently,
\begin{equation}
    \label{eq:degreeCPn}
    d(u)=\frac{1}{4\pi}\int_{\CP^1}u^*\omega_\FS\in\Z.
\end{equation}
By the Hurewicz Theorem,  $d(u)$ also
represents the homotopy class $[u]\in\pi_2(\CP^n)$.

Let $j$ and $I$ denote the complex structures of $\CP^1$ and $\CP^n$,
respectively.  A smooth map $u\colon\CP^1\to\CP^n$ is \emph{holomorphic} if, at all points,
$    du_x\circ j=I\circ du_x$, and \emph{anti-holomorphic} if $du_x\circ j=-I\circ du_x$.
Equivalently, in holomorphic coordinate charts on the domain and target, the
coordinate representation of a holomorphic map is holomorphic.  For a
holomorphic or anti-holomorphic map, 
\eqref{eq:kahler} and \eqref{eq:degreeCPn} give
\begin{equation}
    \label{eq:energyDegreeHolomorphic}
    E(u)=4\pi|d(u)|.
\end{equation}
The next well-known result gives an explicit representation of holomorphic maps, see e.g.\ \cite{Huybrechts2005}. For completeness, and to bring together the above definitions, we give an elementary proof.

\begin{lemma}
    \label{lem:holomorphicHomogeneousRepresentation}
Let $u\colon\CP^1\to\CP^n$ be holomorphic.  Then there are homogeneous
polynomials $P_0,\ldots,P_n\in\C[z_0,z_1]$, all of degree $d(u)$ and with no
common zero in $\C^2\setminus\{0\}$, such that
\begin{equation}
    \label{eq:holomorphicHomogeneousRepresentation}
    u([z_0:z_1])=[P_0(z_0,z_1):\cdots:P_n(z_0,z_1)].
\end{equation}
In particular, the $P_j$ have no nonconstant common factor.  The tuple
$(P_0,\ldots,P_n)$ is unique up to multiplication by a common nonzero
constant.  Conversely, any such tuple of polynomials defines a holomorphic map $u\colon\CP^1\to\CP^n$ with degree $d(u)$.
\end{lemma}

\begin{proof}
The converse statement is clear, so let us prove the first statement, assuming 
that $u$ is nonconstant.  To distinguish the domain and target coordinates,
write $[z_0:z_1]$ on $\CP^1$ and $[Z_0:\cdots:Z_n]$ on $\CP^n$.  Choose the
target coordinates so that the image of $u$ is not contained in the
hyperplane $\{Z_0=0\}$.  Then, away from the finite set $u^{-1}(\{Z_0=0\})$, we can write
\[
    u=[1:f_1:\cdots:f_n]
\]
where the functions $f_j$ are holomorphic; in particular, they are
meromorphic functions on $\CP^1$, hence they are rational.  On the domain
chart
$U_0=\{[z_0:z_1]\in\CP^1:z_0\neq0\}$, write $[z_0:z_1]=[1:w]$ and
choose a common denominator $q$, so that
\[
    f_j(w)=\frac{p_j(w)}{q(w)},
    \qquad j=1,\ldots,n,
\]
for some suitable polynomials $p_j$; we also set $p_0:=q$. After dividing $p_0,\ldots,p_n$ by their greatest common divisor, we can assume that they have no common factors.  
Set $d:=\max_j\deg p_j$ and
\[
    P_j(z_0,z_1):=z_0^d p_j(z_1/z_0),
    \qquad j=0,\ldots,n.
\]
These are homogeneous polynomials of degree $d$ with no common zero
when $z_0\neq0$, because the $p_j$ have no common factor, and they have no
common zero at $[0:1]$, because at least one $p_j$ has degree $d$.
Consequently they have no common zero on $\C^2\setminus\{0\}$ and satisfy
\eqref{eq:holomorphicHomogeneousRepresentation}.

It remains to identify $d$.  On $U_0$, with
$w=z_1/z_0$, we have $p_j(w)=P_j(1,w)$. On $\{p_k\neq0\}$, using the coordinates $\zeta^j=p_j/p_k$ in the target $U_k$,  by \eqref{eq:FubiniStudyPotential}  and \eqref{eq:kahler},  we have
$$u^*\omega_\FS = 2 i \p \bar \p \log \bigg(1+\sum_{j\neq k} |\zeta^j|^2\bigg)= \Delta \log \bigg(1+\sum_{j\neq k} |\zeta^j|^2\bigg) d x \wedge d y,$$ 
where in the last equality we used that
$2i \p \bar \p f =2i  f_{w \bar w} \,d w \wedge d \bar w = \Delta f \, d x \wedge d y$ for any $f$, where $w=x+i y$.
Since $u$ is holomorphic, its pullback commutes with $\p \bar \p$, and as $p_j$ is holomorphic we find
\[
    u^*\omega_\FS
    =\Delta 
      \log\left(\frac{F}{|p_k|^2}\right) d x \wedge d y
    =\Delta \log F \, d x \wedge d y, \qquad F:=\sum_j|p_j|^2.
\]
By continuity this identity holds on all of $\C$.  By the divergence theorem,
\[
    \int_{|w|<R}u^*\omega_\FS
    =\int_{|w|=R}\partial_r\log F\,ds.
\]
Writing $p_j(w)=a_jw^d+O(|w|^{d-1})$, with at least one $a_j\neq0$, a simple calculation shows that 
\[
    \partial_r\log F(Re^{i\theta})
    =\frac{2d}{R}+O(R^{-2}),
\]
uniformly in $\theta$.
Letting $R\to\infty$ gives
$\int_{\CP^1}u^*\omega_\FS=4\pi d$, and hence $d(u)=d$ by
\eqref{eq:degreeCPn}.
\end{proof}

\subsection{An averaging lemma}

In Section \ref{sec:B3CPn}, we will construct competitors which are constant in a small ball $B_\e$ and, on each sphere
$\partial B_r$ for $r>\e$, are obtained by projective transformations of a
given map $u$; these transformations must converge to a constant as
$r\to\e$. See already \eqref{eq:competitorCP2}.

Our construction is inspired by the one in  \cite[Theorem 7.4]{Brezis1986}, which deals with the case $n=1$, although it will be substantially more involved. For $n=1$, the construction does not rely on the whole conformal
group of $\CP^1$, but instead considers only a particular class of M\"obius
transformations, corresponding to dilations.  Such transformations can be
written efficiently as
\begin{equation}
\label{eq:Fta}
    F_t^a([z]):=[P_az+tP_a^\perp z], \qquad 0<t\leq 1,
\end{equation}
where, for $a\in\CP^1$ corresponding to a line in $\C^2$, $P_a$ denotes the
Hermitian projection onto that line.  For instance, if $a=[1:0]$, then
$F_t^a([z_0:z_1])=[z_0:tz_1]$. Note that $F_1^a$ is the identity and
$F_t^a(p)\to a$, $t\rightarrow 0$ for $p \notin \mb P(a^\perp)$.

In $\CP^n$ we define exactly the same family of transformations
\eqref{eq:Fta}, now with $a\in\CP^n$ and $z\in\C^{n+1}$.  They are
biholomorphic projective-linear maps.  There is however a key distinction between $\CP^1$ and $\CP^n$ for $n>1$: in the former case all harmonic maps $\CP^1\to \CP^1$ are (anti-)holomorphic, but this is not the case for maps $\CP^1\to \CP^n, n>1$, see already Theorem \ref{thm:harmonicSphereDichotomyCPn}. For (anti-)holomorphic maps $u$, the energy of $F_t^a\circ u$ is independent of $t$ and is determined by  $d(u)$, cf.\ \eqref{eq:energyDegreeHolomorphic}, but for a general map the energy of $F_t^a \circ u$ depends on $t$.

Thus the calculations in the $\CP^n$ case are necessarily more involved.
Nonetheless, we succeed in obtaining a
tractable quantity by averaging among all possible dilation directions $a$; this natural idea was also used, to some extent, in \cite{Brezis1986}; see also \cite[Lemma 1.2]{Coron1989} for a somewhat different but related idea.
To this effect, let $d\mu$ be the normalized $\tp{U}(n+1)$-invariant
probability measure on $\CP^n$, and let $dm$ be the normalized Haar measure on
$\tp{U}(n+1)$.  For every integrable
$f\colon\CP^n\to\R$ and every $p\in\CP^n$,
\begin{equation}
\label{eq:Unavg}
    \int_{\tp{U}(n+1)}f(gp)\,dm(g)
    =\int_{\CP^n}f(a)\,d\mu(a).
\end{equation}
Indeed, the left-hand side is independent of $p$, and hence equals its
$\mu$-average over $p$; then one changes variables $p\mapsto gp$ and uses the
$\tp{U}(n+1)$-invariance of $\mu$.

We can now state the main technical result of this subsection:

\begin{lemma}
\label{lem:averageFt}
For every $0<t\leq1$, every $p\in\CP^n$, and every $X\in T_p\CP^n$, we have
\begin{align}
\label{eq:avg_ang}
    \int_{\CP^n}|dF_t^a(X)|_h^2\,d\mu(a)&=|X|_h^2,\\
\label{eq:avg_rad}
    \int_{\CP^n}|\partial_tF_t^a(p)|_h^2\,d\mu(a)&=4A_n(t),
\end{align}
where
\begin{equation}
\label{eq:defAn}
    A_n(t):=n\int_0^1
    \frac{\tau(1-\tau)^n}
    {\bigl(\tau+t^2(1-\tau)\bigr)^2}\,d\tau.
\end{equation}
Here $|\cdot|_h$ denotes the norm induced by the Fubini--Study metric
$h$.
\end{lemma}

\begin{proof}
For fixed $t$, consider the averaged two-form
\[
    \omega_t:=\int_{\CP^n}(F_t^a)^*\omega_{\FS}\,d\mu(a).
\]
For $g\in\tp{U}(n+1)$, the projections satisfy
$P_{ga}=gP_ag^{-1}$.  Hence
\begin{equation}
\label{eq:invF}
    F_t^{ga}=g\circ F_t^a\circ g^{-1},
    \qquad g\in\tp{U}(n+1).
\end{equation}
Since $\mu$ is $\tp{U}(n+1)$-invariant, this implies that $\omega_t$ is
$\tp{U}(n+1)$-invariant and hence $\omega_t=c(t)\omega_\FS$ for some $c(t)\in \R$. This is a general fact about so-called irreducible homogeneous spaces, of which $\CP^n$ is a prime example, see \cite[7.41-7.44]{Besse1987}, but it has an elementary proof which we sketch: the isotropy group at $[e_0]$ acts on
$T_{[e_0]}\CP^n\simeq\C^n$ through the standard $\tp{U}(n)$ action,
whose invariant real two-forms are precisely the multiples of the standard K\"ahler form.  Indeed, the skew-adjoint complex-linear map associated with such a form commutes with $\tp{U}(n)$ and is therefore a real multiple of $I$.
Now, since $F_t^a$ is
homotopic to the identity through $F_s^a$, with $s\in[t,1]$, we see that
\[
    [(F_t^a)^*\omega_{\FS}]=[\omega_{\FS}]
    \qquad\text{in }H^2(\CP^n;\R).
\]
Averaging this identity over $a$ gives
$[\omega_t]=[\omega_{\FS}]$, so $c(t)=1$ and $\omega_t=\omega_\FS$.

Because $F_t^a$ is holomorphic, $dF_t^a$ commutes with $I$ and thus, by
\eqref{eq:kahler}, we have
$$|d F_t^a(X)|_h^2= (F_t^a)^*\omega_\FS(X,IX).$$
Thus, as  $\omega_t=\omega_{\FS}$, we have
\[
    \int_{\CP^n}|dF_t^a(X)|_h^2\,d\mu(a)
    =\omega_t(X,IX)
    =\omega_{\FS}(X,IX)
    =|X|_h^2,
\]
which proves \eqref{eq:avg_ang}.

To prove \eqref{eq:avg_rad}, fix $a,p\in\CP^n$ and choose unit lifts
$\hat a,\hat p\in\mb S^{2n+1}\subset\C^{n+1}$, and set
\[
    \gamma(t):=P_a\hat p+tP_a^\perp\hat p,
    \qquad
    s:=|\langle\hat a,\hat p\rangle|.
\]
All unsubscripted norms in the following computation are induced by the
standard Hermitian inner product on $\C^{n+1}$.
Then $[\gamma(t)]=F_t^a(p)$, and
\[
    |\gamma(t)|^2=s^2+t^2(1-s^2),\qquad
    \dot\gamma(t)=P_a^\perp\hat p,\qquad
    |\dot\gamma(t)|^2=1-s^2,\qquad
    \langle\dot\gamma(t),\gamma(t)\rangle=t(1-s^2).
\]
Writing $\zeta=\gamma/|\gamma|$, differentiation gives
\[
    \dot\zeta
    =\frac1{|\gamma|}\left(
      \dot\gamma-
      \frac{\operatorname{Re}\langle\dot\gamma,\gamma\rangle}
           {|\gamma|^2}\gamma
    \right), \qquad 
    (\dot\zeta)^{\mathrm{hor}}
    \overset{\eqref{eq:HopfHorizontalProjection}}{=}\frac1{|\gamma|}
    \left(\dot\gamma
    -\frac{\langle\dot\gamma,\gamma\rangle}{|\gamma|^2}\gamma\right). 
\]
Since $F_t^a(p)=H_n(\zeta(t))$, we have
\[
    |\partial_tF_t^a(p)|_h^2
    \overset{\eqref{eq:HopfDilation}}{=}4\left(
      \frac{|\dot\gamma|^2}{|\gamma|^2}
      -\frac{|\langle\dot\gamma,\gamma\rangle|^2}{|\gamma|^4}    \right)=f_t(s),
    \qquad
    f_t(s):=4\frac{s^2(1-s^2)}
    {\bigl(s^2+t^2(1-s^2)\bigr)^2}.
\]

It remains to average this identity.  By \eqref{eq:invF}, the result is
independent of $p$, so we may take $p=[e_0]$.  By \eqref{eq:Unavg},
\[
    \int_{\CP^n}|\partial_tF_t^a(p)|_h^2\,d\mu(a)
    =\int_{\tp{U}(n+1)}
      f_t(|\langle ge_0,e_0\rangle|)\,dm(g).
\]
The map $g\mapsto ge_0$ pushes forward Haar measure to normalized surface
measure on $\mb S^{2n+1}$.  Spherical coordinates
$(z_0,z')=(se^{i\theta},\sqrt{1-s^2}\,\xi)$ show that $s=|z_0|$ has density
$2ns(1-s^2)^{n-1}\,ds$ on $[0,1]$.  Consequently,
\[
\begin{aligned}
    \int_{\CP^n}|\partial_tF_t^a(p)|_h^2\,d\mu(a)
    &=\int_0^1f_t(s)\,2n s(1-s^2)^{n-1}\,ds\\
    &\overset{\tau=s^2}{=}
      4n\int_0^1
      \frac{\tau(1-\tau)^n}
      {\bigl(\tau+t^2(1-\tau)\bigr)^2}\,d\tau
    =4A_n(t),
\end{aligned}
\]
as claimed.
\end{proof}

We conclude by recording the next elementary lemma:

\begin{lemma}[\cite{Fujioka1993}]
    \label{lem:AmIntegralEstimate}
The functions defined in \eqref{eq:defAn} satisfy
\[
    \left(\int_0^1\sqrt{A_2(t)}\,dt\right)^2<\frac32,
    \qquad
    \left(\int_0^1\sqrt{A_m(t)}\,dt\right)^2<\frac m2
    \quad\text{for every integer }m\geq3.
\]
\end{lemma}

\section{Harmonic maps \texorpdfstring{$\mb S^2\to \CP^n$}{S2 to CPn}}\label{sec:S2toCPn}

For a complex subspace $W\subset\C^{n+1}$, write
\[
    \mb P(W):=\{L\subset W:L\text{ is a complex line}\}.
\]
A map into $\mb P(W)$ is \emph{linearly full} (or simply \emph{full}) if
its image is not contained in $\mb P(V)$ for any proper complex subspace
$V\subset W$.  Given $u\colon\mb S^2\to\CP^n$, we set
\begin{equation}
    \label{eq:projectiveSpanDimension}
    W:=\operatorname{span}_{\C}
       \left(\bigcup_{x\in\mb S^2}u(x)\right),
    \qquad
    m:=\dim_{\C}W-1.
\end{equation}
Thus $\mb P(W)$ is the smallest projective subspace containing the image of
$u$, and $u$ is linearly full when regarded as a map into
$\mb P(W)\cong\CP^m$.

Harmonic spheres into $\CP^n$ can be completely described using twistor theory \cite{EellsWood1983}. For an introduction to these methods in the slightly simpler (but closely related) setting of sphere-valued maps we refer the reader to \cite[\S 6]{Helein2002}.  In particular, we have the following result:

\begin{theorem}[Energy dichotomy for harmonic two-spheres \cite{EellsWood1983}]
    \label{thm:harmonicSphereDichotomyCPn}
Let $u\in C^\infty(\mb S^2;\CP^n)$ be harmonic, and define $m$ by
\eqref{eq:projectiveSpanDimension}.  Then exactly one of the following
holds:
\begin{enumerate}
    \item $u$ is holomorphic or anti-holomorphic, and
    \[
        E_{\mb S^2}(u)=4\pi|d(u)|\geq4\pi m;
    \]
    \item $u$ is neither holomorphic nor anti-holomorphic, and
    \[
        E_{\mb S^2}(u)\geq4\pi(3m-2)\geq16\pi.
    \]
\end{enumerate}
\end{theorem}

This result is surely known to some experts but, since we could not find it stated explicitly anywhere (except the $n=2$ version of the result, which the reader can find in \cite[Corollary 2.3]{Lemaire2002}), we explain here how it can be deduced from the results in \cite{EellsWood1983}.

\begin{proof}[Proof of Theorem \ref{thm:harmonicSphereDichotomyCPn}]
The assertions are immediate if $u$ is constant, so suppose that $u$ is
nonconstant.  After identifying $\mb P(W)$ with $\CP^m$, we may regard $u$
as a linearly full harmonic map into $\CP^m$.  We recall the part of the
directrix theory that we need.  If $f\colon\mb S^2\to\CP^m$ is linearly
full and holomorphic, its associated curves $f_r$, $0\leq r\leq m$, are
the holomorphic extensions of the osculating-space maps
\[
    x\longmapsto
    \operatorname{span}_{\C}\{F(x),\partial_zF(x),\ldots,
    \partial_z^rF(x)\}\in \operatorname{Gr}_{r+1}(\C^{m+1}),
\]
where $F$ is a local holomorphic lift of $f$; see
\cite[Section~3A]{EellsWood1983}.  Corollary~7.4 of
\cite{EellsWood1983} says that there is a unique pair $(f,r)$, with
$r\in\{0,\ldots,m\}$, for which
\[
    u(x)=f_r(x)\cap f_{r-1}(x)^\perp,
    \qquad f_{-1}(x)=\{0\},\quad f_m(x)=\C^{m+1}.
\]
Here $f_{r-1}(x)^\perp$ is the orthogonal complement in $\C^{m+1}$ with
respect to the standard Hermitian inner product
$\langle\cdot,\cdot\rangle$.
More precisely, Eells--Wood identify $r$ with the $D''$-order of $u$ and
$m-r$ with its $D'$-order; see \cite[Theorem~6.9(i) and
Remark~6.10(i)]{EellsWood1983}.  By \cite[Definition~6.1]{EellsWood1983},
the $D''$-order is zero if and only if $u$ is holomorphic, while the
$D'$-order is zero if and only if $u$ is anti-holomorphic.  Hence
\[
    u\text{ is holomorphic}\iff r=0,
    \qquad
    u\text{ is anti-holomorphic}\iff r=m.
\]

For $0\leq a\leq m-1$, the Pl\"ucker embedding is the holomorphic map
\[
    \begin{aligned}
    \operatorname{Pl}_a\colon \operatorname{Gr}_{a+1}(\C^{m+1})
        &\longrightarrow \mb P\bigl(\Lambda^{a+1}\C^{m+1}\bigr),\\
    \operatorname{span}_{\C}\{v_0,\ldots,v_a\}
        &\longmapsto [v_0\wedge\cdots\wedge v_a].
    \end{aligned}
\]
The \emph{Pl\"ucker degree} $D_a$ of $f_a$ is the ordinary projective
degree, as defined in \eqref{eq:degreeCPn}, of the holomorphic curve
$\operatorname{Pl}_a\circ f_a$, i.e.\
\[
    D_a=\frac{1}{4\pi}\int_{\mb S^2}
    (\operatorname{Pl}_a\circ f_a)^*\omega_\FS\in\N_0.
\]
We set $D_{-1}=0$, and $D_m=0$ since
$\operatorname{Gr}_{m+1}(\C^{m+1})$ is a point.

Let $\beta_a\geq0$ denote the total ramification index of $f_a$.  By
\cite[Proposition~7.1]{EellsWood1983}  (with $c=1$, corresponding to our normalization of the
Fubini--Study metric), we have
\begin{equation}
    \label{eq:degreeFormEnergy}
    \frac{E_{\mb S^2}(u)}{4\pi}=D_{r-1}+D_r.
\end{equation}
Since $\mb S^2$ has genus zero, the Pl\"ucker relations for the associated
curves read
\begin{equation}
    \label{eq:pluckerDegrees}
    D_a-2D_{a-1}+D_{a-2}=-2-\beta_{a-1},
    \qquad a=1,\ldots,m;
\end{equation}
see \cite[equation~(7.1)]{EellsWood1983}. For $a=-1,0,\ldots,m$, set
\[
    Q_a:=(a+1)(m-a).
\]
Thus $Q_{-1}=Q_m=0$, and a direct calculation gives
\[
    Q_a-2Q_{a-1}+Q_{a-2}=-2,
    \qquad a=1,\ldots,m.
\]
Consequently $H_a:=D_a-Q_a$ satisfies $H_{-1}=H_m=0$ and, by
\eqref{eq:pluckerDegrees},
\[
    H_a-2H_{a-1}+H_{a-2}=-\beta_{a-1}\leq0,
    \qquad a=1,\ldots,m.
\]
In particular, the successive slopes $H_a-H_{a-1}$ are nonincreasing, so
the piecewise-affine interpolation of the sequence $(H_a)_{a=-1}^m$ is
concave, hence it is non-negative, i.e.\
\begin{equation}
    \label{eq:degreeLowerBoundAssociatedCurves}
    D_a\geq Q_a=(a+1)(m-a),
    \qquad a=-1,0,\ldots,m.
\end{equation}
Finally, \eqref{eq:degreeFormEnergy} and
\eqref{eq:degreeLowerBoundAssociatedCurves} give
\begin{equation}
    \label{eq:lowerBoundFullHarmonicSphere}
    \frac{E_{\mb S^2}(u)}{4\pi}
    \geq r(m-r+1)+(r+1)(m-r)
    =m+2r(m-r)\geq m,
\end{equation}
because $0\leq r\leq m$.

For $r=0$ or $r=m$, \eqref{eq:energyDegreeHolomorphic} gives
$E_{\mb S^2}(u)=4\pi|d(u)|$, while
\eqref{eq:lowerBoundFullHarmonicSphere} gives $|d(u)|\geq m$.  Otherwise
$m\geq2$ and $1\leq r\leq m-1$.  Since
$r(m-r)\geq m-1$, \eqref{eq:lowerBoundFullHarmonicSphere} gives
\[
    \frac{E_{\mb S^2}(u)}{4\pi}
    \geq m+2r(m-r)\geq3m-2\geq4,
\]
which proves the second alternative.
\end{proof}

\section{Tangent maps \texorpdfstring{$\mb B^3\to \CP^n$}{B3 to CPn}}\label{sec:B3CPn}

The purpose of this section is to prove Theorem \ref{thm:Class3d}.  We first
prove the classification when the target is $\CP^2$.  We then use Theorem
\ref{thm:harmonicSphereDichotomyCPn} and the averaged projective dilations from Section~3 to
reduce the general case to $\CP^2$.
Indeed, we begin with:

\begin{theorem}
    \label{thm:CP2Classification}
Let $u\in C^\infty(\mb S^2;\CP^2)$ be a nonconstant harmonic map.  If
$\hat u$ is minimizing, then $u$ is an isometric embedding.
\end{theorem}

We divide the proof of this result into the following three steps.

\begin{proposition}
    \label{prop:step1}
In the setting of Theorem \ref{thm:CP2Classification}, $u$
is holomorphic or anti-holomorphic.
\end{proposition}

\begin{proposition}[\cite{Fujioka1993}]
    \label{prop:step2}
In the setting of Theorem \ref{thm:CP2Classification}, if $u$ is
holomorphic, then $d(u)\leq2$.
\end{proposition}

\begin{proposition}
    \label{prop:step3}
If $u\colon\CP^1\to\CP^2$ is holomorphic and $d(u)=2$, then its
$0$-homogeneous extension $\hat u\colon\mb B^3\to\CP^2$ is not minimizing.
\end{proposition}

\begin{proof}[Proof of Theorem \ref{thm:CP2Classification}]
By Proposition \ref{prop:step1}, the map $u$ is holomorphic or
anti-holomorphic; by composing on the right with an
orientation-reversing isometry, we can assume that we are in the former case.
 Since $u$ is nonconstant, $d(u)\geq1$.  Proposition \ref{prop:step2} gives
$d(u)\leq2$, while Proposition \ref{prop:step3} rules out $d(u)=2$,
therefore $d(u)=1$.  Lemma
\ref{lem:holomorphicHomogeneousRepresentation} gives
\[
    u([z])=[Az],
    \qquad A\in\C^{3\times2},
    \qquad \operatorname{rank}_{\C}A=2.
\]
Thus the image of $u$ is the projective line $\mb P(\operatorname{im}A)$.
Regarding $u$ as a map from $\CP^1$ to that projective line, the
classification of minimizing tangent maps into $\CP^1$ from
\cite{Brezis1986} shows that $u$ is an isometry.  The inclusion of a
projective line in $\CP^2$ is isometric, and hence $u$ is an isometric
embedding. 
\end{proof}

\subsection{The case of \texorpdfstring{$\CP^2$}{CP2}: the first two steps}

Both of the first two steps follow from the same energy estimate. We carry out the calculation for maps into $\CP^m$, with $m\geq2$, since
the estimate will also be used in the reduction to $\CP^2$.
Let $u\in C^\infty(\mb S^2;\CP^m)$ be a harmonic map such that $\hat u$
is minimizing.  Given $\varepsilon\in(0,1)$ and an absolutely continuous
profile $t_\varepsilon\colon[\varepsilon,1]\to[0,1]$ such that
$t_\varepsilon(\varepsilon)=0$, $t_\varepsilon(1)=1$, and
$t_\varepsilon>0$ on $(\varepsilon,1]$, define
\begin{equation}
    \label{eq:competitorCP2}
    v_{a,\varepsilon}(r\omega):=
    \begin{cases}
        a,&0\leq r\leq\varepsilon,\\
        F_{t_\varepsilon(r)}^a(u(\omega)),&\varepsilon<r\leq1.
    \end{cases}
\end{equation}
If $H_a:=\mb P(a^\perp)$, then $u^{-1}(H_a)$ has area zero for almost
every $a\in\CP^m$, by Fubini's theorem.  For such $a$,
$F_{t_\varepsilon(r)}^a(u(\omega))\to a$ for almost every $\omega$ as
$r\downarrow\varepsilon$.  Thus, whenever the annular energy is finite,
dominated convergence identifies the inner trace with the constant $a$.
Gluing to the constant map on $\mb B_\varepsilon$ then gives an admissible
competitor in $W^{1,2}(\mb B^3;\CP^m)$ with trace $u$ on
$\partial\mb B^3$.

\begin{proof}[Proof of Propositions \ref{prop:step1} and \ref{prop:step2}]
Choose the profile implicitly by
\begin{equation}
    \label{eq:profileStepOne}
    \int_0^{t_\varepsilon(r)}\sqrt{A_m(t)}\,dt
    =\frac{1-\varepsilon/r}{1-\varepsilon}
    \int_0^1\sqrt{A_m(t)}\,dt.
\end{equation}
The integral on the right is finite by Lemma \ref{lem:AmIntegralEstimate}.
Since $A_m(t)\geq A_m(1)>0$ for $0<t\leq1$, this defines an absolutely
continuous profile with the required endpoint values.  Differentiating
\eqref{eq:profileStepOne} gives
\[
    \sqrt{A_m(t_\varepsilon(r))}\,t_\varepsilon'(r)
    =\frac{\varepsilon}{(1-\varepsilon)r^2}
    \int_0^1\sqrt{A_m(t)}\,dt.
\]
On the annulus $\varepsilon<r<1$,
\[
    \partial_r v_{a,\varepsilon}
    =t_\varepsilon'(r)\,
    \partial_tF_{t_\varepsilon(r)}^a(u(\omega)),
    \qquad
    \nabla_{\mb S^2}v_{a,\varepsilon}
    =dF_{t_\varepsilon(r)}^a(\nabla_{\mb S^2}u).
\]
Thus, averaging the energy in $a$ and using Lemma \ref{lem:averageFt},
\[
\begin{aligned}
    \int_{\CP^m}E_{\mb B^3}(v_{a,\varepsilon})\,d\mu(a)
    &=(1-\varepsilon)E_{\mb S^2}(u)
      +8\pi\int_\varepsilon^1
      r^2A_m(t_\varepsilon(r))|t_\varepsilon'(r)|^2\,dr\\
    &=(1-\varepsilon)E_{\mb S^2}(u)
      +8\pi\frac{\varepsilon}{1-\varepsilon}
      \left(\int_0^1\sqrt{A_m(t)}\,dt\right)^2.
\end{aligned}
\]
In particular, the averaged energy is finite, so $v_{a,\varepsilon}$ is
an admissible competitor for almost every $a$.  By minimality of
$\hat u$, and since $E_{\mb B^3}(\hat u)=E_{\mb S^2}(u)$, we obtain
\begin{equation}
    \label{eq:oneDimensionalStepTwo}
    \varepsilon E_{\mb S^2}(u)
    \leq 8\pi\frac{\varepsilon}{1-\varepsilon}
    \left(\int_0^1\sqrt{A_m(t)}\,dt\right)^2.
\end{equation}
Dividing by $\varepsilon$ and letting $\varepsilon\to0$ gives
\begin{equation}
    \label{eq:upperEnergyCPn}
    E_{\mb S^2}(u)
    \leq8\pi\left(\int_0^1\sqrt{A_m(t)}\,dt\right)^2.
\end{equation}
Taking $m=2$ and applying Lemma \ref{lem:AmIntegralEstimate}, we obtain
$E_{\mb S^2}(u)<12\pi$.  Theorem
\ref{thm:harmonicSphereDichotomyCPn} therefore implies that $u$ is
holomorphic or anti-holomorphic, proving Proposition \ref{prop:step1}.
If $u$ is holomorphic, \eqref{eq:energyDegreeHolomorphic} gives
$4\pi d(u)=E_{\mb S^2}(u)<12\pi$, and hence $d(u)\leq2$.
This proves Proposition \ref{prop:step2}.
\end{proof}

\subsection{The case of \texorpdfstring{$\CP^2$}{CP2}: the third step}
We now turn to Proposition \ref{prop:step3}, which is much more intricate.  For $A \in \C^{3\times 3}$, define $\Sigma_A$ by the following formula wherever the vector inside the brackets is nonzero:
\begin{gather*}
\Sigma_A([z_0:z_1],[w_0:w_1])
:=
\left[
A\left(
\frac{z_0w_0-z_1w_1}{\sqrt{2}},
\frac{i(z_0w_0+z_1w_1)}{\sqrt{2}},
\frac{z_0w_1+z_1w_0}{\sqrt{2}}
\right)
\right].
\end{gather*}
Note the following immediate consequence of Lemma \ref{lem:holomorphicHomogeneousRepresentation}:

\begin{lemma}
    \label{lemma:VeroneseStandardForm}
    Suppose $u:\CP^1 \to \CP^2$ is holomorphic of degree $d(u)=2$. Then there exists a matrix $A \in \C^{3\times 3}$ such that
    \[
    u([z_0:z_1])=\Sigma_A([z_0:z_1],[z_0:z_1]).
    \]
\end{lemma}

For $A=\Id$, $u$ is called the Veronese embedding. In the following we identify $\CP^1 \cong \mb S^2$; if the precise identification matters, we will specify it.\\
The following lemma allows us to reduce from $\Sigma_A$ to $\Sigma_D$ for a real diagonal $D \in \mb R^{3\times 3}$ in the proof of Proposition \ref{prop:step3}.
\begin{lemma}
    \label{lem:StationarityCPDeg2}
    Suppose $u:\mb S^2 \to \CP^2$ is a holomorphic map of degree $2$ such that $\hat u$ is stationary. Then, up to isometries of $\mb S^2$ and $\CP^2$ we have
    \[
    u([z])=\Sigma_D([z],[z]),
    \]
    where $D=\mathrm{diag}(\sqrt{a},\sqrt{b},\sqrt{c})$, with $0\leq a \leq b\leq c$ and $b,c>0$.
\end{lemma}

\begin{proof}
By Lemma \ref{lemma:VeroneseStandardForm}, there exists
$A\in\C^{3\times 3}$ such that $u([z])=\Sigma_A([z],[z])$. We define
\[
H:=A^*A,
\]
and write this matrix as
\begin{equation}
    \label{eq:SplittingForH}
    H=S+i[k]_\times,
\end{equation}
where $S\in\R^{3\times 3}$ is symmetric and
$[k]_\times x=k\times x$. Let $p(z)$ be the triple defining $\Sigma_{\Id}([z],[z])$, so
\[
p(z_0,z_1)=\frac{1}{\sqrt{2}}(z_0^2-z_1^2,i(z_0^2+z_1^2),2z_0z_1),\quad (z_0,z_1)\in \C^2\setminus \{0\}.
\]
Thus, $Ap(z_0,z_1)$ is a holomorphic lift of $u$, and since $u$ is everywhere defined on $\CP^1\cong S^2$, we have that $Ap(z_0,z_1)\neq 0$ on $\mb C^2\setminus \{0\}$. This implies easily that 
\begin{equation}
    \label{ineq:rkGreater2}
    \mathrm{rank }A\geq 2.
\end{equation}
For unit representatives $z\in\mb S^3$, we identify \(\mathbb{CP}^1\cong\mathbb S^2\) via the Hopf map 
\[
\omega=h(z)=(-2\Re(\bar z_0 z_1),-2\Im(\bar z_0z_1),|z_0|^2-|z_1|^2).
\]
With this identification, a simple computation shows that
\begin{equation}
    \label{eq:FormulaLift}
    p(z)p(z)^*
    =\frac12\bigl(I-\omega\otimes\omega+i[\omega]_\times\bigr).
\end{equation}
Since 
\[
P(\omega):=|Ap(z)|^2=p(z)^*Hp(z)=\tr(Hpp^*),
\]
we derive from \eqref{eq:FormulaLift} and \eqref{eq:SplittingForH} that
\begin{equation}
    \label{eq:FormulaforP}
    P(\omega)=\frac12\big(\tr S-\omega^TS\omega\big)+k\cdot\omega
    =:Q(\omega)+k\cdot\omega,
\end{equation}
where $Q$ is even. Since $|Ap(z)|\neq 0$ for $z \in \C^2 \setminus\{0\}$ we have
\begin{equation}
    \label{QLowerBound}
    Q(\omega)>\lvert k\cdot\omega\rvert
\qquad\text{for all }\omega\in\mathbb S^2.
\end{equation}
Next we claim that
\begin{equation}
    \label{formulaNormIntermsOfLaplace}
    |\nabla u|^2=4+2\Delta_{\mathbb S^2}\log P.
\end{equation}
Indeed, let the affine coordinate be $\zeta=z_1/z_0$ and choose $z=(1,\zeta)/\sqrt{1+|\zeta|^2}$. If $F(\zeta)=Ap(1,\zeta)$, then the
quadratic homogeneity of $p$ gives
\[
|F(\zeta)|^2=(1+|\zeta|^2)^2P(\omega).
\]
Using $\omega_{\CP^2}=2i\p\bar\p\log|Z|^2$ and
$\omega_{\mb S^2}=2i\p\bar\p\log(1+|\zeta|^2)$, together with
$i\p\bar\p f=\frac12(\Delta_{\mb S^2}f)\omega_{\mb S^2}$, we obtain
\begin{equation}
    \label{eq:PullBackForm}
    u^*\omega_{\CP^2}
    =2i\p\bar\p\log|F|^2
    =(2+\Delta_{\mb S^2}\log P)\omega_{\mb S^2}.
\end{equation}
Since $u$ is holomorphic, $|\nabla u|^2\omega_{\mb S^2}=2u^*\omega_{\mathbb{CP}^2}$,
and hence \eqref{eq:PullBackForm} implies \eqref{formulaNormIntermsOfLaplace}.\\
As $\hat u$ is stationary, we have from Proposition  \ref{prop:stationarityHomogeneous} that
\[
\int_{\mb S^2}\omega |\nabla_{\mb S^2} u|^2d\omega=0.
\]
Inserting \eqref{formulaNormIntermsOfLaplace}
and using that $\Delta_{\mathbb S^2}\omega=-2\omega$ we have
\[
0=\int_{\mathbb S^2}\omega\log P(\omega)\,d\omega.
\]
Pairing $\omega$ with $-\omega$ and taking the scalar product with $k$,
we obtain
\[
0=
\frac12\int_{\mathbb S^2}
(k\cdot\omega)
\log\left(
\frac{Q(\omega)+k\cdot\omega}
     {Q(\omega)-k\cdot\omega}
\right)d\omega.
\]
From \eqref{QLowerBound} the integrand is nonnegative and vanishes only when $k \cdot \omega =0$.
Thus $k=0$, so $H$ is real symmetric.\\
Choose $R \in SO(3)$ such that
\[
R^THR=\mathrm{diag}(a,b,c),
\quad 0\le a\le b\le c.
\]
Hence, after precomposing
with a domain isometry, we may assume
\[
A^*A=\mathrm{diag}(a,b,c)=D^2,
\quad
D:=\mathrm{diag}(\sqrt a,\sqrt b,\sqrt c).
\]
By polar decomposition, there exists \(U\in U(3)\) such that \(UA=D\).
After applying the corresponding target isometry, we obtain
\[
u([z])=\Sigma_D([z],[z]).
\]
Finally, we obtain from \eqref{ineq:rkGreater2} that $b,c>0$. This completes the proof. 
\end{proof}

For the rest of this subsection, we identify $\CP^1\cong\mb S^2$ via
\[
[z_0:z_1]\mapsto(2\Re(\bar z_0z_1),2\Im(\bar z_0z_1),|z_0|^2-|z_1|^2),
\]
where $|z_0|^2+|z_1|^2=1$. Since $p(z_0,-z_1)=\mathrm{diag}(1,1,-1)p(z_0,z_1)$, this change of identification is absorbed by a target isometry in the normal form above.

\begin{lemma}
    \label{lem:AsymptoticCompetitor}
Let $u\in C^\infty(\mb S^2;N)$ be a harmonic map such that $\hat u$ is minimizing. Let \(v\in W^{1,2}_{\mathrm{loc}}(\R^3;N)\) and suppose that there are constants
\(C_0,R_0>0\) such that
\begin{equation}\label{eq:decay}
    |v(x)-\widehat u(x)|\leq C_0|x|^{-2},
    \qquad
    |Dv(x)-D\widehat u(x)|\leq C_0|x|^{-3}
\end{equation}
for almost every \(x\in \R^3\setminus B_{R_0}\). Then 
\begin{equation}
    \label{eq:liminfEqualityEnergy}
    \liminf_{R\to\infty}
    \bigl(E_{B_R}(v)-E_{B_R}(\widehat u)\bigr)\geq 0.
\end{equation}
\end{lemma}

\begin{proof}
    Let $R>2R_0$. Denote $v_R(x):=v(Rx)$. From \eqref{eq:decay} we have
    \begin{equation}
        \label{eq:AnnulusBound}
        \|v_R-\hat u\|_{W^{1,\infty}(B_1\setminus B_{1/2})}\leq C R^{-2}.
    \end{equation}
    We pick a cutoff $\chi \in C_c^\infty(B_1)$ such that $\chi \equiv 1$ in an open neighborhood of $\overline {B_{1/2}}$. For $R$ large enough we define
    \[
    w_R(x)=\begin{cases}
        v_R(x),\quad \text{in }B_{1/2},\\
        \exp_{\hat u(x)}(\chi(x)\exp_{\hat u(x)}^{-1}(v_R(x))),\quad \text{in }B_1\setminus B_{1/2}.
    \end{cases}
    \]
    Note that $w_R \in W^{1,2}(B_1;N)$ and it agrees with $\hat u$ on $\p B_1$, hence $w_R$ is an admissible competitor for $\hat u$. Moreover, from \eqref{eq:AnnulusBound} together with the smoothness of the exponential map we have
    \begin{equation}
        \label{energyAnnulusofwR}
        E_{B_1\setminus B_{1/2}}(w_R)=E_{B_1\setminus B_{1/2}}(\hat u)+O(R^{-2}).
    \end{equation}
    Minimality of $\hat u$ implies
    \[
    E_{B_1}(\hat u)\leq E_{B_1}(w_R).
    \]
    Splitting $B_1=B_{1/2}\cup (B_{1}\setminus B_{1/2})$, using that $w_R=v_R$ in $B_{1/2}$ and \eqref{energyAnnulusofwR}, we obtain that
    \begin{equation}
        E_{B_{1/2}}(v_R)-E_{B_{1/2}}(\hat u)\geq -CR^{-2}.
    \end{equation}
    Equation \eqref{eq:liminfEqualityEnergy} now follows easily by rescaling.
\end{proof}

\begin{remark}
    Before we prove Proposition \ref{prop:step3}, let us give some motivation for the competitor and the coordinates used. If $\hat u$ is minimizing, Lemma \ref{lem:StationarityCPDeg2} allows us to suppose that
    \[
    \hat u(x)=\Sigma_{D}\left(\frac{x}{|
    x|},\frac{x}{|x|}\right),\quad  x \in \R^3\setminus \{0\}.
    \]
    Next we use an idea following the proof of Theorem 7.4 in \cite{Brezis1986}: We split the degree $2$ singularity of $\hat u$ at the origin into $2$ degree $1$ singularities. A simple Ansatz for achieving this is to define the competitor
    \[
    v(x)= \Sigma_{D}\left(\frac{x-e_3}{|
    x-e_3|},\frac{x+e_3}{|x+e_3|}\right),\quad  x \in \R^3\setminus \{e_3,-e_3\}.
    \]
    In the proof we will see that the asymptotics as $|x| \to \infty$ of $v-\hat u$ are favorable, so due to Lemma \ref{lem:AsymptoticCompetitor} we do not have to worry about correcting the boundary data.\\
    To compute the energy of $v$, we work in \textit{prolate spheroidal coordinates}, with foci at $\pm e_3$, hence exactly at the singularities of $v$. These coordinates are defined as
    \[
    \begin{cases}
        x_1=\sinh \mu \sin \nu\cos\varphi\\
        x_2=\sinh \mu \sin \nu\sin \varphi\\
        x_3=\cosh \mu \cos \nu,
    \end{cases}
    \]
    where $\mu >0$, $0<\nu <\pi$, $0\leq \varphi <2\pi$. Let us recall some basic facts. First, the Euclidean metric in prolate spheroidal coordinates is given as
    \begin{equation}
        \label{eq:MetricInprolate}
        g_{\R^3}=(\sinh^2\mu +\sin^2\nu)(d\mu^2+d\nu^2)+\sinh^2\mu\sin^2\nu d\varphi^2.
    \end{equation}
    Second, for any $\alpha >0$ the level sets
    \[
    \{\mu =\alpha\}=\{x\in \R^3:|x-e_3|+|x+e_3|=2\cosh \alpha\},
    \]
    are prolate ellipsoids, which enclose the points $\pm e_3$. As $\alpha \downarrow 0$ the ellipsoids degenerate to the segment $[e_3,-e_3]$.\\
    Therefore, for any $\alpha >0$, $v|_{{\{\mu=\alpha\}}}:\mb S^2 \to \CP^2$, viewed in the angular variables $(\nu,\varphi)$ on the standard round sphere, is holomorphic of degree $2$. Thus, the main advantage of working in these coordinates is that for any $\alpha>0$, the angular energy of $v|_{{\{\mu=\alpha\}}}$ can be computed directly from \eqref{eq:energyDegreeHolomorphic}. The computationally difficult part is to show that the remaining part of the energy is sufficiently small. This is carried out in Lemma \ref{lem:EnergyCompR}.
\end{remark}

\begin{proof}[Proof of Proposition \ref{prop:step3}]
    Suppose for the sake of contradiction that $\hat u$ is minimizing. In particular, we have that $\hat u$ is stationary, so by combining Lemma \ref{lemma:VeroneseStandardForm} and Lemma \ref{lem:StationarityCPDeg2}, we can suppose that
    \[
    u([z])=\Sigma_D([z],[z]),
    \]
    where $D=\mathrm{diag}(\sqrt{a},\sqrt{b},\sqrt{c})$, with $0\leq a \leq b\leq c$ and $b,c>0$. \\
    Denote
    \[
        \omega_{\pm}(x):=\frac{x\mp e_3}{|x\mp e_3|}.
    \]
    The Ansatz for the competitor is 
    \begin{equation}
        \label{eq:AnsatzComp}
        v(x):=\Sigma_D(\omega_+(x),\omega_-(x)),\quad x\in \R^3\setminus \{e_3,-e_3\},
    \end{equation}
    where we identified $\mb S^2 \cong \CP^1$ in each argument of $\Sigma_D$.\\
    \textbf{Step 1: Coordinate Expression.} Let $(\mu,\nu,\varphi)$ be the prolate spheroidal coordinates introduced above. For $\omega \in S^2$ we write
    \[
    \omega=(\sin \nu \cos \varphi,\sin \nu \sin \varphi,\cos \nu),
    \]
    and let
    \begin{equation}
        \label{def:XiProj}
        \xi(\nu,\varphi):=(\cos \nu \cos\varphi-i\sin\varphi,\cos\nu\sin\varphi+i\cos\varphi,\sin\nu).
    \end{equation}
    We claim that in these coordinates
    \begin{equation}
        \label{eq:CoordinateVeps}
        v(\mu,\nu,\varphi)=\left[ D(s\xi_1,s\xi_2,\xi_3)\right],\quad s:=\tanh \mu.
    \end{equation}
    Let us denote by $\theta_{\pm}$ the polar angles of $\omega_\pm$ from the $x_3$-axis. A direct computation using the definition of prolate spheroidal coordinates shows that
    \begin{equation}
        \label{eq:FormPolarAngompm}
        \tan\frac{\theta_+}{2}=\coth \frac{\mu}{2}\tan \frac{\nu}{2},\quad \tan\frac{\theta_-}{2}=\tanh \frac{\mu}{2}\tan \frac{\nu}{2}.
    \end{equation}
    A possible Hopf lift of $\omega_\pm$ is given by
    \[
    q_{\pm}=\left(\cos\frac{\theta_\pm}{2},e^{i\varphi}\sin\frac{\theta_\pm}{2}\right) \in \mb S^3.
    \]
    In particular, by \eqref{eq:FormPolarAngompm} we can use the homogeneous representatives
    \[
    q_+=\left(1,e^{i\varphi}\frac{\tan(\nu/2)}{\tanh(\mu/2)}\right),\quad q_-=\left(1,e^{i\varphi}\tan\frac\nu 2\tanh\frac{\mu}{2}\right).
    \]
    Inserting the above into the definition of $\Sigma_D$, we obtain
    \begin{align*}
        &\Sigma_D([q_+],[q_-])=\\
        &\left[\sqrt{a}(1-e^{2i\varphi}\tan^2(\nu/2)):\sqrt{b}i(1+e^{2i\varphi}\tan^2(\nu/2)):\sqrt{c}e^{i\varphi}\tan\frac{\nu}{2}\left(\tanh\frac{\mu}{2}+\frac{1}{\tanh\frac{\mu}{2}}\right)\right] \\
        &=\left[\sqrt{a}(1-e^{2i\varphi}\tan^2(\nu/2)):\sqrt{b}i(1+e^{2i\varphi}\tan^2(\nu/2)):\sqrt{c}\frac{2e^{i\varphi}\tan\frac{\nu}{2}}{s}\right] \\
        &=\left[\sqrt{a}s(1-e^{2i\varphi}\tan^2(\nu/2)):\sqrt{b}si(1+e^{2i\varphi}\tan^2(\nu/2)):\sqrt{c}2e^{i\varphi}\tan\frac{\nu}{2}\right],
    \end{align*}
    where we used the identity
    \[
    \tanh (\mu/2)+\frac{1}{\tanh (\mu/2)}=\frac{2}{s},
    \]
    in the second equality. A simple computation shows that
    \[
    \begin{cases}
        1-e^{2i\varphi}\tan^2(\nu/2)=(1+\tan^2(\nu/2))e^{i\varphi}\xi_1\\
        i(1+e^{2i\varphi}\tan^2(\nu/2))=(1+\tan^2(\nu/2))e^{i\varphi}\xi_2\\
        2e^{i\varphi}\tan(\nu/2)=(1+\tan^2(\nu/2))e^{i\varphi}\xi_3,
    \end{cases}
    \]
    which, after dividing by the common factor, immediately implies \eqref{eq:CoordinateVeps}.\\
    \textbf{Step 2: Regularity} The goal here is to show that $v \in W^{1,2}_{loc}(\R^3;\CP^2)$. The vector in \eqref{eq:CoordinateVeps} is nonzero away from the foci: its third component is nonzero for $0<\nu<\pi$, while at $\nu=0,\pi$ we have $s>0$, $|\xi_2|=1$ and $b>0$. The open axis between the two foci can be written as
    \[
    L:=\{\mu=0,0<\nu<\pi\}=\{(0,0,x_3):-1<x_3<1\}.
    \]
    On $L$, $v$ equals $[0:0:1]$ constantly. Suppose now that $p \in L$. In an open neighborhood around $p$ we have that $\sqrt{c}\xi_3=\sqrt{c}\sin\nu\neq 0$. Hence, we can write $v$ in affine coordinates on this neighborhood as
    \[
    v=(W_1,W_2):=\left(\frac{\sqrt{a}s\xi_1}{\sqrt{c}\xi_3},\frac{\sqrt{b}s\xi_2}{\sqrt{c}\xi_3}\right).
    \]
    From the definition of $\xi$ we obtain
    \[
    W_1= \sqrt{\frac{a}{c}}\frac{\cos \nu x_1 - i x_2}{\cosh \mu \sin^2 \nu},\quad W_2 =\sqrt{\frac{b}{c}}\frac{\cos \nu x_2+i x_1}{\cosh \mu \sin^2 \nu}.
    \]
    By the equalities
    \[
    \cosh\mu =\frac{|x-e_3|+|x+e_3|}{2},\quad \cos \nu =\frac{|x+e_3|-|x-e_3|}{2},
    \]
    the functions $W_1$, $W_2$ are smooth on this coordinate neighborhood. In particular the only possible singularities of the map are $\pm e_3$.\\
    Next we prove that $v \in W^{1,2}$ in a neighborhood of $e_3$. We first claim that
    \[
    \Sigma_D:\CP^1 \times U \to \CP^2,
    \]
    is smooth where $U$ is an open neighborhood of $[1:0]$. Indeed we use the lift
    \[
    \Sigma_D([z],[w])=\left[
D\left(
\frac{z_0w_0-z_1w_1}{\sqrt{2}},
\frac{i(z_0w_0+z_1w_1)}{\sqrt{2}},
\frac{z_0w_1+z_1w_0}{\sqrt{2}}
\right)
\right]=:[Dp(z,w)].
    \]
    For the smoothness, it is enough to show that $Dp(z,(1,0))\neq 0$ for all $z \in \C^2\setminus\{0\}$. But since
    \[
    Dp(z,(1,0))=\frac{1}{\sqrt{2}}(\sqrt{a}z_0,i\sqrt{b}z_0,\sqrt{c}z_1),
    \]
    and $b,c>0$, the vector above cannot vanish for any $z\neq0$, which proves the smoothness. Now note that if we pick a small enough ball $B_\rho(e_3)$, $\omega_-|_{B_{\rho}(e_3)}$ takes values in $U$. Hence $v \in W^{1,2}(B_{\rho}(e_3))$ follows immediately from the chain rule. The argument for $-e_3$ is completely analogous.\\
    \textbf{Step 3: Energy computation} 
    Recall that for any $\alpha>0$, the angular map $(\nu,\varphi)\mapsto v(\alpha,\nu,\varphi)$ on the standard round sphere is holomorphic of degree two and hence by \eqref{eq:energyDegreeHolomorphic}, we have
    \begin{equation}
        \label{eq:SphericalEnergyv}
        \int_{\mb S^2}|\nabla_{\mb S^2}v|^2d\omega =16\pi.
    \end{equation}
    Denote $\Omega_T:=\{\mu < T\}$. From \eqref{eq:MetricInprolate}, we can write the energy for $v$ in prolate coordinates as
    \begin{equation}
        \label{eq:energyInProlateCoords}
        E_{\Om_T}(v)=\frac{1}{2}\int_0^T\sinh\mu\left(\int_{\mb S^2}|\nabla_{\mb S^2}v|^2d\omega\right)d\mu +R_T,
    \end{equation}
    where
    \[
    R_T=\frac12\int_{0}^T\int_{\mb S^2}\left(\sinh \mu|\p_\mu v|^2+\frac{1}{\sinh\mu}|\partial_{\varphi}v|^2\right)d\omega d\mu.
    \]
    If we combine \eqref{eq:energyInProlateCoords} with \eqref{eq:SphericalEnergyv}, we have
    \begin{equation}
        \label{energyofvinEllipsoid}
        E_{\Om_T}(v)=8\pi(\cosh T-1)+R_T,
    \end{equation}
    Let us write $R:=\lim_{T \to \infty}R_T$. Then we obtain from Lemma \ref{lem:EnergyCompR} below that
    \begin{equation}
        \label{Rlesseightpi}
        R<8\pi.
    \end{equation}
    \textbf{Step 4: Conclusion} Since $\Omega_T$ is an ellipsoid with semi-axes $\sinh T$, $\sinh T$ and $\cosh T$, its radial function $\rho_T$ satisfies
    \[
    \rho_T(\omega)=\cosh T +O(e^{-T}),
    \]
    uniformly for $\omega \in S^2$. Hence, we have
    \begin{equation}
        \label{eq:EnergyOfuinEllipsoid}
        E_{\Om_T}(\hat u)=8\pi\cosh T+o(1).
    \end{equation}
    Combining \eqref{energyofvinEllipsoid} with \eqref{eq:EnergyOfuinEllipsoid} we derive
    \begin{equation}
        \label{eq:EnergyDiffUnscaled}
        \lim_{T \to\infty}(E_{\Om_T}(v)-E_{\Om_T}(\hat u))=R-8\pi<0.
    \end{equation}
    Recall the definition of $v$ in \eqref{eq:AnsatzComp}. We write in polar coordinates $x=r\omega$ and denote $t:=r^{-1}$. Let us define
    \[
    \Psi(t,\omega):=\Sigma_D\left(\frac{\omega-te_3}{|\omega-te_3|},\frac{\omega+te_3}{|\omega+te_3|}\right) =v(r\omega).
    \]
    With a similar argument as in Step 2, we can pick $t_0>0$ small such that $\Psi$ is smooth on the compact set $[-t_0,t_0]\times \mb S^2$. Moreover, note that $\Psi(0,\omega)=\hat u(r\omega)$ and $\Psi(\cdot,\omega)$ is even in $t$, by the symmetry of $\Sigma_D$ with respect to its arguments. Hence, the first derivative with respect to $t$ vanishes at $t=0$, and Taylor's theorem gives
    \begin{equation}
        \label{eq:DifferenceOfvandhatu}
        v(r\omega)-\hat u(r\omega)=O(t^2)=O(r^{-2}),\quad\nabla v(r\omega)-\nabla\hat u(r\omega)=O(r^{-3}) ,\quad \text{as }r\to \infty
    \end{equation}
    uniformly in $\omega$. In particular, this implies that the limit in \eqref{eq:EnergyDiffUnscaled} is independent of the exhaustion so
    \begin{equation}
        \label{eq:EnergyDiffBalls}
        \lim_{r \to \infty}(E_{B_r}(v)-E_{B_r}(\hat u))=R-8\pi<0.
    \end{equation}
    From \eqref{eq:DifferenceOfvandhatu}, we can apply Lemma \ref{lem:AsymptoticCompetitor}. Whence,
    \[
        \lim_{r \to \infty}(E_{B_r}(v)-E_{B_r}(\hat u))\geq 0,
    \]
    which contradicts \eqref{eq:EnergyDiffBalls}. This completes the proof.
\end{proof}

\begin{lemma}
    \label{lem:EnergyCompR}
    Using the notation of the proof of Proposition \ref{prop:step3}, we have
    \[
    R<8\pi.
    \]
\end{lemma}

\begin{proof}
    Suppose $W:M \to \C^3\setminus \{0\}$ is a nonzero lift of a projective map defined on some smooth manifold $M$. Then as in the proof of Lemma \ref{lem:averageFt} one has the formula
    \begin{equation}
        \label{eq:AbstractDiffFormula}
        |d[W]|^2=4\left(\frac{|dW|^2}{|W|^2}-\frac{|\langle W,dW\rangle|^2}{|W|^4}\right).
    \end{equation}
    Recall that $v$ has the lift
    \[
    W=(\sqrt{a}s\xi_1,\sqrt{b}s\xi_2,\sqrt{c}\xi_3),
    \]
    with $\xi$ defined in \eqref{def:XiProj}. We easily compute
    \begin{equation}
        \label{eq:DiffFormulas}
        \begin{cases}
            \p_\mu W=(1-s^2)(\sqrt{a}\xi_1,\sqrt{b}\xi_2,0)\\
            \partial_\varphi \xi_1=-\cos\nu\sin\varphi-i\cos\varphi\\
            \p_\varphi \xi_2=\cos\nu\cos\varphi-i\sin\varphi.
        \end{cases}
    \end{equation}
    We introduce the notation
    \[
    L:=a|\xi_1|^2+b|\xi_2|^2,\quad M:=a|\xi_2|^2+b|\xi_1|^2.
    \]
    Applying \eqref{eq:AbstractDiffFormula} together with \eqref{eq:DiffFormulas} we compute
    \begin{align}
    \label{eq:DerivVInMu}
        |\p_\mu v|^2&=4\left(\frac{|\p_\mu W|^2}{s^2L+c|\xi_3|^2}-\frac{|\langle W,\p_\mu W \rangle|^2}{(s^2L+c|\xi_3|^2)^2}\right) \nonumber \\
        &=4\left(\frac{(1-s^2)^2L}{s^2L+c|\xi_3|^2}-\frac{(1-s^2)^2 s^2 L^2}{(s^2L+c|\xi_3|^2)^2}\right)\nonumber \\
        &=\frac{4(1-s^2)^2cL|\xi_3|^2}{(s^2L+c|\xi_3|^2)^2}.
    \end{align}
    In a similar way, one computes
    \begin{equation}
        \label{eq:DerivVInPhi}
        |\p_\varphi v|^2=\frac{4s^2|\xi_3|^2(s^2ab|\xi_3|^2+cM)}{(s^2L+c|\xi_3|^2)^2}.
    \end{equation}
    Recall the definition
    \[
    R_T=\frac12\int_{0}^T\int_{\mb S^2}\left(\sinh \mu|\p_\mu v|^2+\frac{1}{\sinh\mu}|\partial_{\varphi}v|^2\right)d\omega d\mu.
    \]
    We first insert \eqref{eq:DerivVInMu} and \eqref{eq:DerivVInPhi} to obtain
    \[
    R_T=2\int_0^T\int_{\mb S^2}\frac{|\xi_3|^2}{(s^2L+c|\xi_3|^2)^2}\left[\sinh\mu (1-s^2)^2cL+\frac{s^2}{\sinh\mu}(s^2ab|\xi_3|^2+cM)\right]d\omega d\mu.
    \]
    Next, we change variables in the integral above, with 
    \[
    s=\tanh\mu,\quad d\mu=\frac{ds}{1-s^2}.
    \]
    Together with the equality, $\sinh\mu=s/\sqrt{1-s^2}$, we obtain after taking $T \to \infty$, the equation
    \begin{equation*}
        R=\lim_{T \to \infty}R_T=2\int_0^1\frac{s}{\sqrt{1-s^2}}\int_{\mb S^2}\frac{|\xi_3|^2\left[(1-s^2)cL+s^2ab|\xi_3|^2+cM\right]}{(s^2L+c|\xi_3|^2)^2}d\omega ds.
    \end{equation*}
    Changing variables a second time with $y=\sqrt{1-s^2}$, $sds=-\sqrt{1-s^2}dy$ we derive
    \begin{equation}
        \label{eq:DefOfRPrecise}
        R=2\int_0^1\int_{\mb S^2}\frac{|\xi_3|^2\left[y^2cL+(1-y^2)ab|\xi_3|^2+cM\right]}{((1-y^2)L+c|\xi_3|^2)^2}d\omega dy.
    \end{equation}
    Next we want to find a splitting of the form $R=8\pi+2I$, for some negative integral $I$. To do so, we compute
    \begin{equation}
        \label{directDifferentiationy}
        \frac{d}{dy}\left(\frac{yc|\xi_3|^2}{(1-y^2)L+c|\xi_3|^2}\right)=\frac{c|\xi_3|^2(L+c|\xi_3|^2+y^2L)}{((1-y^2)L+c|\xi_3|^2)^2},
    \end{equation}
    and note that for $|\xi_3|^2>0$, we can evaluate
    \begin{equation}
        \label{derivativeEval}
        \frac{yc|\xi_3|^2}{(1-y^2)L+c|\xi_3|^2}\bigg|_{y=1}=1,\quad \frac{yc|\xi_3|^2}{(1-y^2)L+c|\xi_3|^2}\bigg|_{y=0}=0. 
    \end{equation}
    Moreover, we have the splitting
    \begin{align}
        \label{splittingOfNumerator}
        |\xi_3|^2(y^2&cL+(1-y^2)ab|\xi_3|^2+cM)\\
        &=c|\xi_3|^2(L+c|\xi_3|^2+y^2L)+|\xi_3|^2\left[c(M-L)-(c^2-ab)|\xi_3|^2-y^2ab|\xi_3|^2\right].
    \end{align}
    Since the integrand in \eqref{eq:DefOfRPrecise} is nonnegative, we can change the order of integration. Then we use the splitting \eqref{splittingOfNumerator}, the identity \eqref{directDifferentiationy} and the evaluation \eqref{derivativeEval} to obtain
    \begin{align*}
        R&= 2\int_0^1\int_{\mb S^2}\frac{|\xi_3|^2\left[y^2cL+(1-y^2)ab|\xi_3|^2+cM\right]}{((1-y^2)L+c|\xi_3|^2)^2}d\omega dy \\
        &=2\int_{\mb S^2}\left(\frac{yc|\xi_3|^2}{(1-y^2)L+c|\xi_3|^2}\bigg|_{y=1}- \frac{yc|\xi_3|^2}{(1-y^2)L+c|\xi_3|^2}\bigg|_{y=0}\right)d\omega\\
        &+2 I\\
        &=8\pi +2I,
    \end{align*}
    where we defined
    \[
    I:=\int_0^1\int_{\mb S^2}\frac{|\xi_3|^2\left[c(M-L)-(c^2-(1-y^2)ab)|\xi_3|^2\right]}{((1-y^2)L+c|\xi_3|^2)^2}d\omega dy.
    \]
    Hence, the proof is completed if we can show that
    \begin{equation}
        \label{eqInequalityforI}
        I<0.
    \end{equation}
    Recall the definition of $M$ and $L$. We can compute them explicitly to see that
    \begin{equation}
        \label{explicitMandL}
        \begin{cases}
            L&=a(\cos^2\nu\cos^2\varphi+\sin^2\varphi)+b(\cos^2\nu\sin^2\varphi+\cos^2\varphi)\\
            M&=a(\cos^2\nu\sin^2\varphi+\cos^2\varphi)+b(\cos^2\nu\cos^2\varphi+\sin^2\varphi).
        \end{cases}
    \end{equation}
    Therefore for fixed $\nu$, the replacement $\varphi \mapsto \varphi+\pi/2$ interchanges $M$ and $L$. In particular, if we denote the integrand in the definition of $I$ as $F(y,\nu,\varphi)$, we have for fixed $y \in (0,1)$ and $\nu \in (0,\pi)$ the equality
    \begin{equation}
        \label{splittingThroughIntegrand}
        \int_{0}^{2\pi}F(y,\nu,\varphi)d\varphi=\frac{1}{2}\int_{0}^{2\pi}(F(y,\nu,\varphi)+F(y,\nu,\varphi+\pi/2))d\varphi.
    \end{equation}
    Denoting
    \[
    A:=(1-y^2)L+c|\xi_3|^2,\quad B:=(1-y^2)M+c|\xi_3|^2
    \]
    one computes
    \begin{align}
        \label{eq:ExpressionOfIntegrand}
        F(y,\nu,\varphi)&+F(y,\nu,\varphi+\pi/2) \nonumber
        \\&=\frac{|\xi_3|^2}{A^2B^2}\big(c(1-y^2)(M-L)^2(A+B)-(c^2-(1-y^2)ab)(A^2+B^2)|\xi_3|^2\big)
    \end{align}
    We record 3 inequalities.
    \begin{enumerate}
        \item \label{ineq:it1} $A^2+B^2 \geq c |\xi_3|^2(A+B)$ which follows directly from $A,B\geq c|\xi_3|^2$.
        \item \label{ineq:it2} $(M-L)^2\leq (b-a)^2|\xi_3|^4$ which follows from the identity $M-L=(a-b)|\xi_3|^2\cos(2\varphi)$.
        \item \label{ineq:it3} $c^2-(1-y^2)(ab+(b-a)^2)\geq 0$ which follows as
        \[
        ab+(b-a)^2\leq a^2-ab+b^2\leq b^2\leq c^2,
        \]
        from the ordering $a\leq b\leq c$.
    \end{enumerate}
    Applying \ref{ineq:it1}--\ref{ineq:it3} to \eqref{eq:ExpressionOfIntegrand} one immediately derives that
    \begin{equation}
        \label{nonnegativePairedInt}F(y,\nu,\varphi)+F(y,\nu,\varphi+\pi/2)\leq 0.
    \end{equation}
    To complete the proof of \eqref{eqInequalityforI}, we have to see that the inequality is strict on a set of positive measure. For this, note that if $y \in (0,1)$ and $\nu \in (0,\pi)$ and moreover, $|\xi_1|,|\xi_2|>0$ we have that $|\xi_3|^2,L,M>0$. This implies that $A,B>c|\xi_3|^2>0$ and therefore,
    \[
    A^2+B^2> c|\xi_3|^2(A+B),
    \]
    meaning \ref{ineq:it1} is strict. Since the set where $|\xi_1|=0$ or $|\xi_2|=0$ is a set of measure $0$, \ref{ineq:it1} is strict on a set of positive measure. In particular \eqref{nonnegativePairedInt} is strict on a set of positive measure and hence, \eqref{eqInequalityforI} holds. This completes the proof.
\end{proof}

\subsection{Reduction to \texorpdfstring{$\CP^2$}{CP2} and proof of
Theorem \ref{thm:Class3d}}

\begin{proposition}
    \label{prop:CP2Reduction}
Let $u\in C^\infty(\mb S^2;\CP^n)$ be a nonconstant harmonic map.  If
$\hat u$ is minimizing, then the projective span of $u(\mb S^2)$ has
complex dimension at most two.
\end{proposition}

\begin{proof}
Let $W$ and $m$ be defined by \eqref{eq:projectiveSpanDimension}.  
After a target isometry, we may identify $\mb P(W)$ with $\CP^m$ and regard
$u$ as a linearly full harmonic map $\mb S^2\to\CP^m$.  The map $\hat u$
is still minimizing with respect to competitors taking values in this
subspace.  Theorem \ref{thm:harmonicSphereDichotomyCPn} therefore gives
\begin{equation}
\label{eq:lowerEnergy}
E_{\mb S^2}(u)\geq 4\pi m.
\end{equation}

If $m\geq3$, \eqref{eq:upperEnergyCPn} and Lemma
\ref{lem:AmIntegralEstimate} give
\[
    E_{\mb S^2}(u)
    \leq8\pi\left(\int_0^1\sqrt{A_m(t)}\,dt\right)^2
    <4\pi m,
\]
contradicting \eqref{eq:lowerEnergy}.  Hence $m\leq2$.
\end{proof}

\begin{proof}[Proof of Theorem \ref{thm:Class3d}]
Let $W$ and $m$ be defined by \eqref{eq:projectiveSpanDimension}.  Since
$u$ is nonconstant, $m\geq1$, while
Proposition \ref{prop:CP2Reduction} gives $m\leq2$.
If $m=1$, then $u$ takes values in the projective line $\mb P(W)$ and the
classification of minimizing tangent maps into $\CP^1$
\cite{Brezis1986} shows that
$u\colon\mb S^2\to\mb P(W)$ is an isometry.  Since projective lines are
totally geodesic in $\CP^n$, $u$ is an isometric embedding.  If
$m=2$, the same conclusion follows directly from Theorem
\ref{thm:CP2Classification}.  
\end{proof}

\section{Tangent maps \texorpdfstring{$\mb B^4\to \CP^n$}{B4 to CPn} with isolated singularities}\label{sec:B4CPn}

The purpose of this section is to prove Theorem \ref{thm:Class4d}. The proof is split into two main steps: in Theorem \ref{thm:HopfFactorizationCPn} we show that $u$ can be factored through the Hopf map, and then we apply the coarea formula and Theorem \ref{thm:Class3d} to conclude.

\subsection{Horizontally weakly conformal maps}

For the reader's convenience we recall the definitions of horizontal weak conformality
\cite[Definition~2.4.2]{Baird2003} and pseudo-horizontal weak conformality
\cite[Definition~8.2.3]{Baird2003}, together with the Fuglede--Ishihara
characterization of harmonic morphisms \cite[Theorem~4.2.2]{Baird2003}.

\begin{definition}
\label{def:HWC}
A map $u\in C^\infty(M,N)$ is \textit{horizontally weakly conformal} if, at every
$x\in M$, either $du_x=0$, or $du_x$ is surjective and its restriction
to the horizontal space
\[
    \mc H_x:=(\ker du_x)^\perp
\]
is conformal.  In terms of the adjoint, this is equivalent to
\[
    du_x\circ du_x^*
    =\lambda(x)^2\Id_{T_{u(x)}N}
    =\frac{|du_x|^2}{\dim N}\Id_{T_{u(x)}N}.
\]

If $(N,h,I)$ is almost Hermitian, then $u$ is
\textit{pseudo-horizontally weakly conformal} if
\begin{equation}
    \label{eq:PseudoHorizontalWeakConformality}
    \bigl[du_x\circ du_x^*,I_{u(x)}\bigr]=0
    \qquad\text{for every }x\in M.
\end{equation}
The first condition implies the second.  When $\dim_\R N=2$, the two
conditions are equivalent.

A \textit{harmonic morphism} is a map which  is harmonic and horizontally weakly conformal.
\end{definition}

We begin with a non-trivial generalization of a beautiful calculation due to Rivière \cite{Riviere2023}. 
He showed that any harmonic map $u\in C^\infty(\mb S^3;\mb S^2)$ with index 4 must be horizontally weakly conformal. Note that, by Theorem \ref{thm:Xin}, any non-constant smooth harmonic map must have index \textit{at least} 4. First, we replace the index assumption with the weaker assumption that one has an appropriate lower bound on $\lambda_1(u)$, rather than an upper bound on the index; second, we allow the domain to be a higher dimensional sphere, and the target to be a general complex projective space.

\begin{theorem}
    \label{thm:Riv}
    Let $m\geq 3$ and let $u \in C^\infty(\mb S^m;\CP^n)$ be a harmonic map such that $$2-m\leq \lambda_1(u).$$
    Then $u$ is pseudo-horizontally weakly conformal.
\end{theorem}

\begin{proof}
If $u$ is constant there is nothing to prove, so we assume that $u$ is
nonconstant.
Let $g$ denote the round metric on $\mb S^m$.  
In the notation of Theorem \ref{thm:Xin}, set $X^k := X_{e_k}$ for simplicity, where $\{e_k\}_{k=1}^{m+1}$ is the standard basis of $\R^{m+1}$.  We also set, for $k=1, \dots, m+1,$
$$
\phi^k:=du(X^k),\qquad
\psi^k:=I \phi^k.
$$

A simple but crucial observation in our proof is that the connection Laplacian commutes with $I$.  Indeed, since
$\CP^n$ is K\"ahler, its complex structure is parallel: $\nabla I=0$.
Consequently, for every smooth
section $X$ of $u^*T\CP^n$, if we choose a local orthonormal frame
$(e_\alpha)$ which is parallel at a fixed but arbitrary point then, at that point,  using
\eqref{eq:ConnectionLaplacian}, we obtain
\begin{equation}
\label{eq:parallel}
\nabla^*\nabla(IX)
=-\sum_\alpha\nabla_{e_\alpha}\nabla_{e_\alpha}(IX)
=-\sum_\alpha I\nabla_{e_\alpha}\nabla_{e_\alpha}X
=I\nabla^*\nabla X.
\end{equation}
Thus, we have
\begin{align*}
[J_u + (m-2)]\psi^k
&= \nabla^* \nabla(I\phi^k)
   - \sum_\a R(I\phi^k,u_\a)u_\a
   + (m-2)I\phi^k\\
&\overset{\eqref{eq:parallel}}{=}
  I\bigl(\nabla^*\nabla\phi^k+(m-2)\phi^k\bigr)
  -\sum_\a R(I\phi^k,u_\a)u_\a\\
&\overset{\eqref{eq:Xin}}{=}
  \sum_\a\left(
  I R(\phi^k,u_\a)u_\a-R(I\phi^k,u_\a)u_\a
  \right).
\end{align*}
Using \eqref{eq:curvsimple}, we can now compute
\[
I R(X,Y)Y-R(IX,Y)Y
=-\left(h(X,Y)IY+h(X,IY)Y\right),
\]
and so
\[
[J_u+(m-2)]\psi^k
=-\sum_\a\left(
h(\phi^k,u_\a)Iu_\a+h(\phi^k,Iu_\a)u_\a
\right).
\]
Hence, by \eqref{eq:ParsevalConformalFields}, we find
\begin{align*}
\sum_k h([J_u+(m-2)]\psi^k,\psi^k)
&=\sum_{k,\a}\left(
h(\phi^k,Iu_\a)^2-h(\phi^k,u_\a)^2
\right)\\
&\overset{\eqref{eq:ParsevalConformalFields}}{=}\sum_{\a,\beta}\left(
h(u_\beta,Iu_\a)^2-h(u_\a,u_\beta)^2
\right)=2|u^*\omega_\FS|^2-|u^*h|^2.
\end{align*}
Here we used \eqref{eq:kahler} and the standard norm on two-forms, so that
\begin{equation}
\label{eq:norm2form}
2|u^*\omega_\FS|^2=\sum_{\a,\beta}
(u^*\omega_\FS)(e_\a,e_\beta)^2.
\end{equation}
Thus our assumption $2-m\leq\lambda_1(u)$ implies that $J_u + (m-2)$ is positive semidefinite, so
\begin{equation}
    \label{eq:PHWCIntegralInequality}
    0\leq
    \int_{\mb S^m}
    \left(2|u^*\omega_\FS|^2-|u^*h|^2\right).
\end{equation}
For each $x\in\mb S^m$, set $A_x:=du_x\circ du_x^*$.  The next lemma,
applied to $L=du_x$, gives
\[
2|u^*\omega_\FS|^2-|u^*h|^2
=-\frac12|A_x+IA_xI|^2.
\]
It follows from \eqref{eq:PHWCIntegralInequality} that
\[
0\leq-\frac12\int_{\mb S^m}|A_x+IA_xI|^2\leq0.
\]
Therefore $A_x+IA_xI=0$ at every $x$, or equivalently
$A_xI=IA_x$, as wished.
\end{proof}

\begin{lemma}
Let $(H,h,I,\omega)$ be a Hermitian vector space, let $E$ be a real
inner product space, and let $L\colon E\to H$ be linear.   Then
\begin{equation}
    \label{eq:PHWCLinearAlgebra}
    |L^*h|^2-2|L^*\omega|^2
    =\frac12|A+IAI|^2,\qquad A:=L L^*\colon H\to H,
\end{equation}
where the norm on two-forms is as in \eqref{eq:norm2form}.
\end{lemma}

\begin{proof}
The pullback metric $L^*h$ is represented on $E$ by the self-adjoint
operator $L^* L$, whereas the pullback two-form $L^*\omega$ is
represented by the skew-adjoint operator $L^* I L$.  Therefore
\[
|L^*h|^2=|A|^2,
\qquad
2|L^*\omega|^2
=|L^* I L|_{\mathrm{HS}}^2
=\langle IA,AI\rangle_{\mathrm{HS}},
\]
where $\langle\cdot,\cdot\rangle$ is the usual Hilbert--Schmidt inner product.
Then
\begin{align*}
\frac12|A+IAI|^2
=|A|^2+\langle A,IAI\rangle_{\mathrm{HS}}=|L^*h|^2-2|L^*\omega|^2
\end{align*}
since $|IAI|=|A|$ and
$\langle A,IAI\rangle_{\mathrm{HS}}
=-\langle IA,AI\rangle_{\mathrm{HS}}$.
\end{proof}

When $m=3$, the conclusion of Theorem \ref{thm:Riv} is stronger:

\begin{corollary}[Horizontal conformality for $m=3$]\label{cor:HWC}
Suppose that $m=3$ in Theorem \ref{thm:Riv}. Then, at every $x\in\mb S^3$, either $du_x=0$, or
$\operatorname{rank}du_x=2$, the space $\operatorname{im}du_x$ is a complex
line, and
\[
    du_x\big|_{(\ker du_x)^\perp}\colon
    (\ker du_x)^\perp\longrightarrow\operatorname{im}du_x
\]
is conformal.  In particular, if $n=1$, then $u$ is horizontally weakly
conformal.
\end{corollary}

\begin{proof}
By Theorem \ref{thm:Riv}, the linear map
$    A_x:=du_x\circ du_x^*\colon T_{u(x)}\CP^n\to T_{u(x)}\CP^n$
commutes with $I$.  Hence $\operatorname{im}A_x=\operatorname{im}du_x$ is
$I$-invariant.  Since $\operatorname{rank}du_x\leq3$, either $du_x=0$, or
$\operatorname{rank}du_x=2$ and $\operatorname{im}du_x$ is a complex line.
In the latter case, since the restriction of $A_x$ to $\operatorname{im}du_x$ is
non-negative, self-adjoint and complex-linear, it is a non-negative scalar multiple of the
identity, which gives the claim.
\end{proof}

Using the theory of harmonic morphisms together with a transversality argument, we can now use Theorem \ref{thm:Riv} to prove:

\begin{theorem}
    \label{thm:HopfFactorizationCPn}
Let $u\in C^\infty(\mb S^3;\CP^n)$ be a nonconstant harmonic map with
$$-1 \leq \lambda_1(u).$$  Then there is an
isometry $Q\colon\mb S^3\to\mb S^3$ and a holomorphic map
$f\colon\CP^1\to\CP^n$ such that
\begin{equation}
    \label{eq:HopfFactorizationCPn}
    u=f\circ H\circ Q.
\end{equation}
\end{theorem}

\begin{proof}
We first reduce to the case $n=1$, so assume $n\geq2$.  To perform this
reduction, we use Sard's theorem to show that a generic projective
$(n-2)$-plane in $\CP^n$ is disjoint from the lower-dimensional set
$u(\mb S^3)$.  Consider the Grassmannian
$\operatorname{Gr}_{n-1}(\C^{n+1})$ of complex $(n-1)$-dimensional
subspaces of $\C^{n+1}$.  Any element
$W\in\operatorname{Gr}_{n-1}(\C^{n+1})$ determines a projective $(n-2)$-plane
$\mathbb P(W)\subset\CP^n$.  Viewing points of $\CP^n$ as complex lines,
define the manifold
\[
    \mc M:=\{(\ell,W)\in\CP^n\times
    \operatorname{Gr}_{n-1}(\C^{n+1}):\ell\subset W\}.
\]
Then $\mc M\to\CP^n$ is a smooth fibre bundle with fibre
$    \operatorname{Gr}_{n-2}(\C^{n+1}/\ell)
    \cong \operatorname{Gr}_{n-2}(\C^n),$ which has real dimension $4(n-2)$.
Now define the pullback bundle
\[
	u^*\mc M
    =\{(x,W)\in\mb S^3\times\operatorname{Gr}_{n-1}(\C^{n+1}):
    u(x)\in\mathbb P(W)\},
\]
with its projection
$q\colon u^* \mc M\to\operatorname{Gr}_{n-1}(\C^{n+1})$.
 It follows that
\[
    \dim_\R u^* \mc M
    =\dim_\R\mb S^3+4(n-2)
   =4n-5<4n-4=\dim_\R\operatorname{Gr}_{n-1}(\C^{n+1}).
\]
Thus every point of $u^* \mc M$ is a critical point of $q$, hence by Sard's Theorem the set
$$q(u^* \mc M)=\{W\in\operatorname{Gr}_{n-1}(\C^{n+1}):
\mb P(W)\cap u(\mb S^3)\neq \emptyset\}$$
has measure zero.

 Since $u$ is nonconstant, we find $x_0$ with
$du_{x_0}\neq0$.  By Corollary \ref{cor:HWC},
$\operatorname{im}du_{x_0}$ is a complex line.  Using the action of
$\tp{U}(n+1)$ and the affine coordinates
\eqref{eq:AffineCoordinatesCPn}, we may therefore suppose that
\[
    u(x_0)=[1:0:\cdots:0],
    \qquad
    \operatorname{im}du_{x_0}
    =\operatorname{span}_{\C}\{\partial_{w^1}\}.
\]
For $W\in\operatorname{Gr}_{n-1}(\C^{n+1})$, let
\[
    P_W\colon\CP^n\setminus\mathbb P(W)\longrightarrow
    \mathbb P(\C^{n+1}/W)\simeq\CP^1,
    \qquad
    P_W([z])=[z\bmod W],
\]
be the projective linear projection. If we set
$W_0:=\operatorname{span}_{\C}\{e_2,\ldots,e_n\}$, where $\{e_j\}_{j=0}^{n}$ is the standard basis of $\C^{n+1}$,  then
\[
    P_{W_0}([z^0:\cdots:z^n])=[z^0:z^1],
\]
so, in the coordinates \eqref{eq:AffineCoordinatesCPn}, $P_{W_0}$ is the
map $(w^1,\ldots,w^n)\mapsto w^1$.  In particular,
\[
    u(x_0)\notin\mathbb P(W_0),
    \qquad
    d(P_{W_0})_{u(x_0)}(\partial_{w^1}) =\p_{w^1}\neq0.
\]
By continuity, both conditions continue to hold for every $W$ in a nonempty open
neighbourhood $U\subset\operatorname{Gr}_{n-1}(\C^{n+1})$ of $W_0$ and,
since $q(u^* \mc M)$ has measure
zero, we may choose $W\in U\setminus q(u^* \mc M)$.  For this choice,
$\mathbb P(W)\cap u(\mb S^3)=\emptyset$.  Set
\[
    v:=P_W\circ u\colon\mb S^3\to\CP^1.
\]
By the choice of $W$ we have $dv_{x_0}\neq0$;
in particular, $v$ is nonconstant.

Since $u$ is harmonic and pseudo-horizontally weakly conformal by Theorem
\ref{thm:Riv}, and $P_W$ is holomorphic between K\"ahler manifolds,
\cite[Proposition~8.2.4(ii)]{Baird2003} shows directly that $v$ is
a harmonic morphism. We have thus reduced the argument to the case of a map into $\CP^1$.  

The
classification of harmonic morphisms $\mb S^3\to \CP^1$
\cite[Theorem~6.7.7 and Remark~6.7.8]{Baird2003} gives an isometry
$Q\colon\mb S^3\to\mb S^3$ and a nonconstant holomorphic map
$g\colon\CP^1\to\CP^1$ such that
\begin{equation}
    \label{eq:ProjectedHopfFactorization}
    v=g\circ H\circ Q.
\end{equation}
We claim that $u$ is constant on the fibres of $H\circ Q$.  Indeed, at a
point $x$ for which $dg_{H(Q(x))}\neq0$, the map $dv_x$ has rank two.
Corollary \ref{cor:HWC} gives $\operatorname{rank}du_x\leq2$, and hence
$dv_x=d(P_W)_{u(x)}\circ du_x$ implies that $d(P_W)_{u(x)}$ is injective on
$\operatorname{im}du_x$.  Consequently, if
$X\in\ker d(H\circ Q)_x$, then $dv_x(X)=0$ and therefore $du_x(X)=0$.
As the branch set of $g$ is finite, by continuity
$du$ vanishes on $\ker d(H\circ Q)$ everywhere, and  thus $u$ is constant on
each fibre, which is a circle. It follows that
\[
    f(y):=u(x),\qquad x\in(H\circ Q)^{-1}(y),
\]
is a well-defined map $f\colon\CP^1\to\CP^n$ such that
$u=f\circ H\circ Q$ and which is smooth, since $H\circ Q$ is a submersion.

Since $H\circ Q$ is surjective, \eqref{eq:ProjectedHopfFactorization} gives
$P_W\circ f=g$.  Corollary \ref{cor:HWC} and the horizontal conformality of
the Hopf map show that, wherever $df\neq0$, the map $f$ is weakly conformal and
$\operatorname{im}df$ is a complex line.  Hence $df$ is either
complex-linear or complex-antilinear.  However, the
second possibility is excluded by the identity $dg=dP_W\circ df$, because both $dP_W$ and $dg$ are complex-linear, and $g$ has a finite branch set.
Thus $f$ is holomorphic, as wished.
\end{proof}

\subsection{Proof of Theorem \ref{thm:Class4d}}

Recall that $H\colon\mb S^3\to\mb S^2= \CP^1$ is the Hopf map and define
$\widetilde H\colon\B^4\to\B^3$ by
\[
    \widetilde H(x)=|x|^2H\left(\frac{x}{|x|}\right).
\]
The next lemma, although simple-looking, is crucial, and it relies fundamentally on the fact that the Hopf map has constant dilatation, cf.\ \eqref{eq:HopfDilation}.

\begin{lemma}
    \label{lem:ProjectionEnergy}
    Let $V\in W^{1,2}(\B^3;N)$.  Then
    $U:=V\circ\widetilde H$ belongs to $W^{1,2}(\B^4;N)$ and
    \[
    \int_{\B^4}|dU|^2\,dx
    =\pi\int_{\B^3}|dV|^2\,dy.
    \]
\end{lemma}

\begin{proof}
    The following computation holds almost everywhere by the Sobolev chain
    rule.  Write $x=r\omega$.  At $\omega\in\mb S^3$, choose an orthonormal
    frame $(e_1,e_2,e_3)$ adapted to $H$, so that
    \[
        dH(e_1)=2E_1,\qquad dH(e_2)=2E_2,\qquad dH(e_3)=0,
    \]
    where $(H(\omega),E_1,E_2)$ is an orthonormal basis of $\R^3$; see, for instance, \cite[\S 2.3]{GuerraLamyZemas2025} for further details on these calculations.
    Since $\widetilde H(r\omega)=r^2H(\omega)$, the chain rule gives
    \[
        dU(\partial_r)
        =2r\,dV_{r^2H(\omega)}(H(\omega))
    \]
    and, for the spherical differential,
    \[
        d_\omega U(r,\cdot)(e_i)
        =\begin{cases}
        2r^2\,dV_{r^2H(\omega)}(E_i),&i=1,2,\\
        0,&i=3.
        \end{cases}
    \]
    It follows that
    \begin{equation}
        \label{eq:ComputationOfGradu}
        |dU(r\omega)|^2
        =4r^2|dV(r^2H(\omega))|^2.
    \end{equation}
    The coarea formula applied to the Hopf map gives
    \begin{equation}
        \label{eq:CoareaHopf}
        \int_{\mb S^3}F(H(\omega))\,d\omega
        =\frac{\pi}{2}\int_{\mb S^2}F(\xi)\,d\xi.
    \end{equation}
    Therefore, using \eqref{eq:ComputationOfGradu},
    \begin{align*}
        \int_{\B^4}|dU|^2\,dx
        &=4\int_0^1r^5\int_{\mb S^3}
          |dV(r^2H(\omega))|^2\,d\omega\,dr\\
        &=2\pi\int_0^1r^5\int_{\mb S^2}
          |dV(r^2\xi)|^2\,d\xi\,dr=\pi\int_{\B^3}|dV|^2\,dy,
    \end{align*}
    where the last equality follows by changing the radial variable from
    $r$ to $r^2$.
\end{proof}

As a consequence of the previous lemma and Theorem \ref{thm:Class3d}, we obtain:

\begin{corollary}\label{cor:minfactor}
    Let $f\colon\CP^1\to\CP^n$ be a nonconstant smooth map and set
    \[
        \hat u(x)=f\circ H\left(\frac{x}{|x|}\right),
        \qquad x\in\B^4\setminus\{0\}.
    \]
    If $\hat u$ is minimizing in $\B^4$, then $f$ is an
    isometric embedding.
\end{corollary}

\begin{proof}
    Define the $0$-homogeneous map $V\colon\B^3\to\CP^n$ by
    \[
        V(y)=f\left(\frac{y}{|y|}\right)
    \]
and note that   $\hat u=V\circ\widetilde H$. Now let  $W\in W^{1,2}(\B^3;\CP^n)$ have the same trace as
    $V$ on $\partial\B^3$, and set
    $w:=W\circ\widetilde H$.  By Lemma
    \ref{lem:ProjectionEnergy}, $w\in W^{1,2}(\B^4;\CP^n)$ is an admissible competitor and so
    \[
        \pi\int_{\B^3}|dV|^2
        =\int_{\B^4}|d\hat u|^2
        \leq\int_{\B^4}|dw|^2
        =\pi\int_{\B^3}|dW|^2
    \]
    where the inequality follows by minimality of $\hat u$.
    Thus $V\colon \B^3\to\CP^n$ is a minimizing tangent map
    and the conclusion follows from Theorem \ref{thm:Class3d}.
\end{proof}

\begin{proof}[Proof of Theorem \ref{thm:Class4d}]
    Minimality of $\hat u$ implies stability.  Hence Proposition
    \ref{prop:stabilitycharact}, applied with $m=3$, gives
    \[
        -1 \leq \lambda_1(u).
    \]
    Theorem \ref{thm:HopfFactorizationCPn} now provides an isometry
    $Q\colon\mb S^3\to\mb S^3$ and a nonconstant holomorphic map
    $f\colon\CP^1\to\CP^n$ such that
    \[
        u=f\circ H\circ Q.
    \]
    The isometry $Q$ is the restriction of an orthogonal map in $\R^4$ that
    we still denote by $Q$.  Since
    $\hat u\circ Q^{-1}$ is still minimizing, and has link $f\circ H$,
    Corollary \ref{cor:minfactor} implies that $f$ is an isometric
    embedding.  Taking $\iota=f$ proves the theorem.
\end{proof}

\subsection{Proof of Theorem \ref{thm:HopfMin}}
The case of $n=1$ was proven in Theorem 5.1 \cite{Coron1989}. Without loss of generality let us denote the boundary trace by the Hopf map
\[
    H(z^0,z^1)=[z^0e_0+z^1 e_1]
\]
for points $(z_0,z_1)\in S^3\subset\C^2$ and $\{e_0,e_1\}$ fixed orthonormal vectors of $\C^{n+1}$.
The strategy of the proof is as follows: for any Sobolev map $v\in W^{1,2}(\B^4;\CP^n)$, we first show the following averaging formula. For the anti-linear isometry $J:\C^{n+1}\to \C^{n+1}$, (w.l.o.g. $n$ is odd)
\[
    J(z^0, z^1, z^2,z^3,...)=(-\overline{z^1},\overline{z^0},-\overline{z^3},\overline{z^2},...)\quad J^2=-\id
\]
and a vector $a\in \C^{n+1}$, denote the projection onto the complex space $\text{span}(a,Ja)$ by $R_a$. Then, define the map $G_t^a:\CP^n\to\CP^n$ by
\[
    G_t^a([z])=[R_az+t(\id-R_a)z]
\]
in other words, the component in the projective line $\mathbb{P}(a\oplus Ja)$ remains untouched by $G_t^a$, whereas the component in the projected orthogonal complement $\mathbb{P}(\langle a,Ja\rangle^\perp)$ gets multiplied by the factor $t$. Using the notation \eqref{eq:Fta}, we can also write
\[
    G_t^a = F_t^a \circ F_t^{Ja}.
\]
Our first claim is the following averaging formula,
\medskip

\textbf{Claim:}
    \[
    \int_{\CP^n}E(G_t^a\circ v)d\mu(a)=E(v).
\]
\textbf{Proof of the claim:} As in the proof of Lemma \ref{lem:averageFt}, we first show that
\[
\int_{\CP^n}(G_t^a)^*\omega_{\mathrm{FS}}\,d\mu(a)=\omega_{\mathrm{FS}}.
\]
Write $A_t^a=R_a+t(\mathrm{id}-R_a)$, so that $G_t^a([z])=[A_t^az]$. By \eqref{eq:FubiniStudyPotential}, \eqref{eq:FubiniStudyCoordinates}, \eqref{eq:kahler},
\[
(G_t^a)^*\omega_{\mathrm{FS}}=\omega_{\mathrm{FS}}+2i\partial\bar\partial\log\frac{|A_t^az|^2}{|z|^2}.
\]
For $a,z\in S^{2n+1}$, since $a\perp Ja$, we have
\[
|A_t^az|^2=t^2+(1-t^2)|R_az|^2
=t^2+(1-t^2)\bigl(|\langle a,z\rangle|^2+|\langle a,Jz\rangle|^2\bigr).
\]
Since $(z,Jz)$ is an orthonormal pair, by $U(n+1)$-invariance of $\mu$, \eqref{eq:Unavg}, 
\[
[z]\longmapsto\int_{\CP^n}\log\frac{|A_t^az|^2}{|z|^2}\,d\mu(a)
\]
is constant. Averaging the preceding identity therefore gives
\[
\int_{\CP^n}(G_t^a)^*\omega_{\mathrm{FS}}\,d\mu(a)=\omega_{\mathrm{FS}}.
\]
Since $G_t^a$ is holomorphic, the same argument as in the proof of Lemma \ref{lem:averageFt} yields
\[
\int_{\CP^n}|dG_t^a(X)|^2\,d\mu(a)=|X|^2
\]
for every $X\in T\CP^n$. Applying this identity to $dv$ and using Fubini,
\[
\int_{\CP^n}E(G_t^a\circ v)\,d\mu(a)=E(v),
\]
as claimed.

Next, for appropriately chosen parameters $a$ and $t_k\to 0$, the maps 
\[
    v_k:= G^a _{t_k}\circ v:\B^4 \to \CP^n
\]
satisfy
\[
    \begin{cases}
        E(v_k)\leq E(v)\\
        v_k\rightharpoonup w\\
        w:\B^4\to\CP^1\subset\CP^n\\
        w\vert_{S^3}= H\quad \text{up to isometry}
    \end{cases}
\]
Indeed, for any sequence $t_k\to 0$, combining Fatou's lemma and the averaging identity,
\[
    \int_{\CP^n}\liminf\limits_{k\to\infty}E(G_{t_k}^a\circ v)d\mu(a)\leq E(v).
\]
Next, for any $a\in\CP^n$, the set of $a$ with $R_a(p)=0$ has measure zero (it is contained in the projective hyperplane orthogonal to $a$). By Fubini, we can choose $a$ such that
\[
    R_a e_0\neq 0, \quad R_a v\neq 0 a.e. \textbf{ on }\B^4, \quad \liminf\limits_jE(G_{t_k}^a)\leq E(v).
\]
For these choices $v_k:= G_{t_k}^a$, up to subsequence,
\[
    v_k\rightharpoonup w \text{ weakly in }W^{1,2}
\]
and given almost everywhere pointwise convergence and the condition $R_a v\neq 0 a.e. \text{ on }\B^4$, as $t_k\to 0$, 
\[
    v_k(a)\to [R_a v(x)]\text{ for almost every }x\in \B^4.
\]
In particular the map $w$ takes values in the projective line $\mathbb{P}(a\oplus Ja)\cong \CP^1$. It remains to check the boundary datum. Since by assumption $v|_{S^3}=H$, for $(z^0,z^1)\in S^3$ we have

\begin{align*}
v_k|_{S^3}(z^0,z^1)
&=G_{t_k}^a(H(z^0,z^1))\\
&=\left[R_a(z^0e_0+z^1e_1)+t_k(\id-R_a)(z^0e_0+z^1e_1)\right]\\
&\longrightarrow [z^0R_ae_0+z^1R_ae_1]=[z^0R_ae_0+z^1JR_ae_0],
\end{align*}
where we used $R_aJ=JR_a$ and $e_1=Je_0$. Since
\[
|R_a(z^0e_0+z^1e_1)|=|R_ae_0|>0
\]
on $S^3$, the convergence is uniform. By weak continuity of the trace operator, $w|_{S^3}$ is the Hopf map up to isometry.
Finally, using the minimizing property of the Hopf cone into $\CP^1$ (Theorem 5.1 \cite{Coron1989}),
\[
E(H)\leq E(w)\leq \liminf\limits_k E(v_k)\leq E(v).
\]
which completes the proof of the Theorem. \qed
\let\oldthebibliography\thebibliography
\let\endoldthebibliography\endthebibliography
\renewenvironment{thebibliography}[1]{
  \begin{oldthebibliography}{#1}
    \setlength{\itemsep}{0.5pt}
    \setlength{\parskip}{0.5pt}
}
{
  \end{oldthebibliography}
}

{\small
\bibliographystyle{abbrv-andre}
\bibliography{sources.bib}
}

\end{document}